\documentclass[a4paper,11pt,twoside]{amsart}

\usepackage[left=2.7cm,right=2.7cm,top=3.5cm,bottom=3cm]{geometry}

\newfont{\cyr}{wncyr10 scaled 1100}

\usepackage{amsmath,amssymb,amsthm,amscd,amsfonts}
\usepackage{mathrsfs}
\usepackage{latexsym}
\usepackage{graphicx}
\usepackage[all]{xy}

\newtheorem{thm}{Theorem}[section]
\newtheorem{pro}[thm]{Proposition}
\newtheorem{cor}[thm]{Corollary}
\newtheorem{lem}[thm]{Lemma}

\theoremstyle{remark}
\newtheorem{rem}[thm]{Remark}

\theoremstyle{remark}
\newtheorem*{rem-no-num}{Remark}

\theoremstyle{remark}

\theoremstyle{remark}

\theoremstyle{remark}

\theoremstyle{definition}
\newtheorem{dfn}[thm]{Definition}

\theoremstyle{definition}

\newcommand{\Q}{\mbox{$\mathbb Q$}}

\newcommand{\T}{\mbox{$\mathbb T$}}
\newcommand{\Z}{\mbox{$\mathbb Z$}}
\newcommand{\N}{\mbox{$\mathbb N$}}
\newcommand{\F}{\mbox{$\mathbb F$}}
\newcommand{\PP}{\mbox{$\mathbb P$}}
\newcommand{\n}{\mbox{$\mathfrak n$}}

\newcommand{\p}{\mbox{$\mathfrak p$}}
\newcommand{\E}{\mbox{$\mathscr E$}}
\newcommand{\D}{\mbox{$\mathscr D$}}
\newcommand{\cO}{\mbox{$\mathcal O$}}
\newcommand{\cC}{\mbox{$\mathcal C$}}
\newcommand{\gal}[2]{\mbox{$\mathrm{Gal}(#1/#2)$}}

\newcommand{\longmono}{\mbox{$\lhook\joinrel\longrightarrow$}}

\newcommand{\longepi}{\mbox{$\relbar\joinrel\twoheadrightarrow$}}

\begin{document}

\title[Torsion points on elliptic curves and a theorem of Igusa]{Torsion points on elliptic curves over function fields and a theorem of Igusa}
\author[A. Bandini, I. Longhi and S. Vigni]{Andrea Bandini, Ignazio Longhi and Stefano Vigni}
\address{A. B.: Dipartimento di Matematica, Universit\`a della Calabria, Via P. Bucci -- Cubo 30B, 87036 Arcavacata di Rende (CS), Italy}
\email{bandini@mat.unical.it}
\address{I. L.: Dipartimento di Matematica, Universit\`a di Milano, Via C. Saldini 50, 20133 Milano, Italy}
\email{longhi@mat.unimi.it}
\address{S. V.: Dipartimento di Matematica, Universit\`a di Milano, Via C. Saldini 50, 20133 Milano, Italy}
\email{stevigni@mat.unimi.it}

\begin{abstract}
If $F$ is a global function field of characteristic $p>3$, we employ Tate's theory of analytic uniformization to give an alternative proof of a theorem of Igusa describing the image of the natural Galois representation on torsion points of non-isotrivial elliptic curves defined over $F$. Along the way, using basic properties of Faltings heights of elliptic curves, we offer a detailed proof of the function field analogue of a classical theorem of Shafarevich according to which there are only finitely many $F$-isomorphism classes of admissible elliptic curves defined over $F$ with good reduction outside a fixed finite set of places of $F$. We end the paper with an application to torsion points rational over abelian extensions of $F$.
\end{abstract}

\maketitle

\section{Introduction}

In any modern treatment of the theory of elliptic curves over arithmetically interesting fields a central role is played by the structure of the subgroup of torsion points viewed as a Galois module. Indeed, if $E$ is an elliptic curve defined over a global field $F$ (by which we mean, as usual, a finite extension of the field $\Q$ of rational numbers or the function field of a smooth, projective algebraic curve over a finite field) then the absolute Galois group $\gal{F^s}{F}$ of $F$ (with $F^s$ being the separable closure of $F$ in a fixed algebraic closure $\bar{F}$) acts on the $n$-torsion subgroup $E[n]$ of $E$ for all integers $n\geq1$ not divisible by the characteristic of $F$. (We remark that $E[n]\subset E(F^s)$ if $\text{char}(F)\nmid n$; this is due to the fact that $E[n]$, viewed as a group scheme, is \'etale over $F$, cf. \cite[Ch. 7, Theorem 4.38]{li}. An alternative proof of this rationality result is given in Proposition \ref{sep-torsion-pro} below.) This basic property, an immediate consequence of the fact that the group law on $E$ is given by rational functions with coefficients in $F$ (which says that, in a fancier language, $E$ is an algebraic group over $F$), naturally leads to the study of one of the most important objects that can be attached to an elliptic curve: its $\ell$-adic representation.

Explicitly, let $\ell$ be a rational prime number such that $\ell\not=\text{char}(F)$. The natural action of $\gal{F^s}{F}$ on the subgroups $E[\ell^n]$ gives rise to (continuous) Galois representations
\begin{equation} \label{red-n-rep-eq}
\bar{\rho}_{E,\ell^n}: \gal{F^s}{F} \longrightarrow \text{Aut}(E[\ell^n])\cong GL_2(\Z/\ell^n\Z),
\end{equation}
the non-canonical isomorphism on the right depending on the choice of a basis of the free module $E[\ell^n]$ over $\Z/\ell^n\Z$. Let now
\[ T_\ell(E):=\varprojlim_n E[\ell^n]\cong\Z_\ell\times\Z_\ell \]
be the \emph{$\ell$-adic Tate module} of $E$, where the projective limit is taken with respect to the multiplication-by-$\ell$ maps. By considering the action of $\gal{F^s}{F}$ on $T_\ell(E)$ we obtain a continuous Galois representation
\begin{equation} \label{l-adic-rep-eq}
\rho_{E,\ell}: \gal{F^s}{F} \longrightarrow \text{Aut}(T_\ell(E))\cong GL_2(\Z_\ell)
\end{equation}
which is called the \emph{$\ell$-adic representation} of $E_{/F}$. Moreover:
\[ \rho_{E,\ell}\bmod{\ell^n}=\bar{\rho}_{E,\ell^n} \]
for all $n\geq1$. Observe that composing $\rho_{E,\ell}$ with the natural inclusion $\Z_\ell\subset\Q_\ell$ gives a representation of $\gal{F^s}{F}$ over a field of characteristic zero. In the following we will regard the representations defined in \eqref{red-n-rep-eq} and \eqref{l-adic-rep-eq} as matrix-valued; in other words, for all $\ell$ we fix a $\Z_\ell$-basis of $T_\ell(E)$. Put $p:=\text{char}(F)>0$; since
\begin{equation} \label{projlim-tors-eq}
\varprojlim_{(n,p)=1}E[n]=\prod_{\ell\not=p}T_\ell(E),
\end{equation}
we have a single large, continuous Galois representation
\begin{equation} \label{rep-eq}
\rho_E: \gal{F^s}{F} \longrightarrow \prod_{\ell\not=p}GL_2(\Z_\ell)
\end{equation}
whose $\ell$th component is the $\ell$-adic representation $\rho_{E,\ell}$. Now denote by $E_\text{$(p)$-tors}$ the subgroup of torsion points of $E$ whose order is prime to $p$. Our choice of bases for the groups $E[n]$ gives an identification
\begin{equation} \label{aut-tors-eq}
\text{Aut}(E_\text{$(p)$-tors})=\varprojlim_{(n,p)=1}GL_2(\Z/n\Z)=\prod_{\ell\not=p}GL_2(\Z_\ell).
\end{equation}
If $F$ has characteristic zero then $E_\text{$(p)$-tors}$ is the whole torsion subgroup $E_\text{tors}$ of $E$, and the inverse limits and the products in \eqref{projlim-tors-eq}, \eqref{rep-eq} and \eqref{aut-tors-eq} are taken over all positive integers and over all prime numbers, respectively.

Given an elliptic curve $E_{/F}$ as above, at least two natural, closely related questions arise:
\begin{equation*}
\text{\emph{Can we describe the image of $\rho_{E,\ell}$ (resp. $\rho_E$) in $GL_2(\Z_\ell)$ (resp. in the product)?}} \tag{$\ast$}
\end{equation*}
And, somewhat less ambitiously:
\begin{equation*}
\text{\emph{How ``large'' is the image of $\rho_{E,\ell}$ (resp. $\rho_E$)?}} \tag{$\ast\ast$}
\end{equation*}
Seeking an answer to questions $(\ast)$ and $(\ast\ast)$ has been the fuel for much investigation in arithmetic algebraic geometry over the past few decades. In particular, we have come to realize that the zero and positive characteristic settings lead to different phenomena. When $F$ is a number field (i.e., a finite extension of $\Q$), a celebrated theorem of Serre gives a very satisfying answer for a large class of elliptic curves $E_{/F}$. Namely,
\begin{thm}[Serre] \label{serre-thm}
Let $F$ be a number field and let $E_{/F}$ be an elliptic curve without complex multiplication\footnote{Recall that, by definition, this means that $\text{End}(E)=\Z$. We remark that throughout our paper we write $\text{End}(E)$ for $\text{End}_{\bar F}(E)$.}. The closed subgroup $\rho_E(\gal{\bar F}{F})$ is open (i.e., has finite index) in $\mathrm{Aut}(E_{\mathrm{tors}})$. Equivalently:
\begin{itemize}
\item[\emph{i)}] $\rho_{E,\ell}(\gal{\bar F}{F})$ is open (i.e., has finite index) in $GL_2(\Z_\ell)$ for all primes $\ell$;
\vskip 2mm
\item[\emph{ii)}] $\rho_{E,\ell}(\gal{\bar F}{F})=GL_2(\Z_\ell)$ for all but finitely many primes $\ell$.
\end{itemize}
\end{thm}
A complete proof of this result, often cited in the literature as ``the open image theorem'', can be found in \cite{se1} (see also \cite{se2}, where part \emph{i)} was first proved).
\begin{rem} \label{serre-rem}
Theorem \ref{serre-thm} is false for CM elliptic curves. In fact, if $E_{/F}$ has complex multiplication \emph{over} $F$ then it can be shown that the action of $\gal{\bar F}{F}$ on $T_\ell(E)$ is abelian, hence $\rho_{E,\ell}$ falls short from being surjective. For details, see the proof of \cite[Ch. II, Theorem 2.3]{si3}.
\end{rem}
In the positive characteristic case (i.e., when $F$ is a global function field) things have a similar but slightly more involved description. To explain what happens, we need to introduce some notation. Let $\cC_{/\mathbb F_r}$ be a geometrically irreducible, smooth, projective algebraic curve over a finite field of characteristic $p>0$, and denote $F:=\F_r(\cC)$ and $\cO_{\mathcal C}$ the function field and the structure sheaf of $\cC$, respectively. (In particular, the irreducibility condition implies that $\F_r$ is algebraically closed in $F$, i.e., $\F_r$ is the field of constants of $F$; see, e.g., \cite[Ch. 3, Corollary 2.14 (d)]{li}.) Fix a closed point $\infty$ of $\cC$ and denote $A:=\cO_{\mathcal C}(\cC-\{\infty\})$ the Dedekind domain of the elements of $F$ that are regular outside $\infty$. The choices of the prime $\infty$ and of the ring $A$, which is essentially the analogue of the ring of integers in an algebraic number field and whose role in our arguments will become apparent only later, are immaterial for the statement of Igusa's results. However, for the sake of clarity, we deem it convenient to introduce our setup once and for all at the beginning of the paper.

The basic example to keep in mind is the following:
\begin{itemize}
\item $\cC=\PP^1_{/\mathbb F_r}$,
\vskip 2mm
\item $\infty=[1:0]$ (the usual point at infinity),
\vskip 2mm
\item $A=\F_r[T]$,
\vskip 2mm
\item $F=\F_r(T)$.
\end{itemize}
If they find it preferable, in all that follows the readers can interpret our notation according to the dictionary above without impairing their understanding in any significant way. In the sequel we adopt the notation of \cite[\S 7.2]{bro}.

Let $n$ be an integer prime to $p=\text{char}(F)$; composing the Galois representation
\[ \bar{\rho}_{E,n}:G_F:=\gal{F^s}{F}\longrightarrow\text{Aut}(E[n])\cong GL_2(\Z/n\Z) \]
with the determinant
\[ \text{det}:\text{Aut}(E[n])\longrightarrow(\Z/n\Z)^\times \]
induces a homomorphism $G_F\rightarrow(\Z/n\Z)^\times$. Set $H_n:=\langle r\rangle\subset(\Z/n\Z)^\times$ for the cyclic
subgroup generated by $r$. The natural identification of $r\in(\Z/n\Z)^\times$ with the $r$th-power Frobenius shows that
$H_n\cong\gal{\F_r(\boldsymbol\mu_n)}{\F_r}$, where $\boldsymbol\mu_n$ denotes the $n$th roots of unity in $F^s$. As in \cite[\S
7.2]{bro}, define the subgroup $\Gamma_n$ of $GL_2(\Z/n\Z)$ via the short exact sequence
\begin{equation} \label{gamma-n-eq}
0\longrightarrow SL_2(\Z/n\Z)\longrightarrow\Gamma_n \xrightarrow{\text{det}}H_n\longrightarrow0.
\end{equation}
In other words, $\Gamma_n$ is the inverse image of $H_n$ in $GL_2(\Z/n\Z)$ via the determinant. Passing to the inverse limit over all integers $n$ not divisible by $p$ we get an exact sequence of profinite groups
\begin{equation} \label{gamma-hat-eq}
0\longrightarrow SL_2(\hat{\Z}_{(p)})\longrightarrow\hat\Gamma\longrightarrow\hat{H}\longrightarrow0.
\end{equation}
Here $\hat{\Z}_{(p)}:=\prod_{\ell\not=p}\Z_\ell$ is the prime-to-$p$ profinite completion of $\Z$, the group $\hat{\Gamma}$ is closed in $GL_2(\hat{\Z}_{(p)})$ and $\hat{H}$ is the subgroup of $\hat{\Z}^\times_{(p)}$ which is topologically generated by $r$. If we perform the same procedure restricting instead to the powers of a prime $\ell\not=p$, the sequence \eqref{gamma-n-eq} yields a sequence
\[ 0\longrightarrow SL_2(\Z_\ell)\longrightarrow\hat\Gamma_\ell\longrightarrow\hat{H}_\ell\longrightarrow0. \]
Equivalently:
\[ \hat\Gamma=\prod_{\ell\not=p}\hat\Gamma_\ell\subset\prod_{\ell\not=p}GL_2(\Z_\ell). \]
Finally, assume that $E_{/F}$ is not isotrivial, i.e. that we cannot find a finite extension $F'$ of $F$ such that $E_{/F'}$ is isomorphic to a constant curve (i.e., a curve defined over the field of constants of $F'$). This is easily seen to be equivalent to the condition $j(E)\notin\bar{\F}_r$.
\begin{thm}[Igusa] \label{igusa-thm}
The profinite group $\rho_E(G_F)$ is an open subgroup of $\hat{\Gamma}$.
\end{thm}
If $E[\ell^\infty]$ is the $\ell$-primary part of the torsion of $E$ (for $\ell$ a prime number), Theorem \ref{igusa-thm} can be equivalently formulated as
\begin{thm} \label{igusa2-thm}
The profinite group $\gal{F(E[\ell^\infty])}{F}$ is an open subgroup of $\hat{\Gamma}_\ell$ for all prime numbers $\ell\not=p$, and is equal to $\hat{\Gamma}_\ell$ for almost all such $\ell$.
\end{thm}
Theorem \ref{igusa-thm}, which is the counterpart in the function field setting of Theorem \ref{serre-thm}, was first proved by Igusa in \cite{i}, where Galois-theoretic techniques are combined with an explicit (but somewhat involved) case-by-case analysis of degenerations of elliptic curves and ramification of fields of modular functions. The language adopted by Igusa is that of pre-Grothendieck algebraic geometry, and this old-fashioned style could probably make his paper hard to appreciate for the ``modern'' reader, typically acquainted (at least at an introductory level) with scheme theory but perhaps less familiar with the geometric formalism of Weil's school.\\

The main goal of our article is to provide an alternative proof of Theorem \ref{igusa-thm}. Our strategy is based on the following simple fact: at the cost of passing to a finite separable extension of $F$, we can choose $\infty$ to be a prime of split multiplicative reduction for $E$ \footnote{I.e., the reduced curve $E\bmod\infty$ has a node as its only singularity, and the slopes of the two tangent lines at the curve in the node belong to the (finite) residue field of $F$ at $\infty$ (and not just to its quadratic extension).}. This implies that $E$ is a \emph{Tate curve} locally at $\infty$, that is, if $F_\infty$ is the completion of $F$ at $\infty$ then the $\infty$-adic Lie group $E(\bar{F}_\infty)$ is Galois-equivariantly isomorphic over $F_\infty$ to the quotient $\bar{F}^\times_\infty/\langle q\rangle$ for a certain ``period'' $q\in F^\times_\infty$. Following ideas of Serre in characteristic zero, this analytic property allows us to replace the algebro-geometric arguments of Igusa with local $\infty$-adic considerations. Although our strategy follows that of Serre closely, it is important to stress a remarkable difference between the characteristic zero and the positive characteristic settings: while our proof is valid for all non-isotrivial elliptic curves over a function field $F$ as above, Tate's theory over number fields yields a proof only of a part of Serre's theorem. In fact, not every elliptic curve without complex multiplication has a non-integral $j$-invariant, so only a proper subclass of non-CM elliptic curves can be dealt with by means of these ``analytic'' arguments (actually, a full proof of Theorem \ref{serre-thm} is achieved in \cite{se1} using different and, in many respects, more sophisticated techniques).

In any case, the fact that this approach to questions $(\ast)$ and $(\ast\ast)$ is fruitful both in characteristic zero and in positive characteristic should come as no surprise. In fact, broadly speaking, this is just another manifestation of the strong parallel between the arithmetic over number fields and the arithmetic over function fields: from this perspective, our work is no exception to this familiar principle. In this direction, the reader is referred to the survey articles \cite{bo} by B\"ockle and \cite{u} by Ulmer for further number fields/function fields analogies on more advanced and abstract topics.\\

Our paper is organized as follows. {\bf Section \ref{faltings-sec}} begins with a review of the basic properties of Faltings heights of elliptic curves over global function fields. These heights, originally defined by Faltings on suitable moduli spaces of abelian varieties of arbitrary dimension, play a key role in the proof of the function field analogue (Theorem \ref{shafarevich-thm}) of the well-known theorem of Shafarevich asserting that there are only finitely many $K$-isomorphism
classes of elliptic curves defined over a number field $K$ with good reduction outside a fixed finite set of places of $K$.
Remarkably enough, the need to use the theory of heights of elliptic curves seems to be peculiar of our characteristic $p$
setting, since in this case the classical diophantine arguments given over number fields (see \cite[Ch. IV, \S 1.4]{se2}) do not apply (cf. Remark \ref{shafarevich-rem}). Although the validity of Shafarevich's theorem for admissible elliptic curves over global function fields is certainly well known (cf., e.g., \cite{f3}, \cite{sz}), to the best of our knowledge this is the first time that the proof is written in a detailed and essentially self-contained way; in this sense, \S \ref{faltings-heights-subsec} and \S \ref{shafarevich-subsec} may be of independent interest and apparently fill a gap in the literature. A crucial role in the proof of Shafarevich's theorem is played by a remarkable property of admissible elliptic curves over global function fields: their Faltings height is bounded in terms of the degree of their conductor and the genus of $F$. A complete proof of this result, which is expected to be valid for elliptic curves over all global fields and is commonly known as the ``height conjecture'', is given in \S \ref{admissible-subsec} (Proposition \ref{height-conj-pro}). The section closes with irreducibility results for Tate modules (\S \ref{irreducibility-subsec}) that are applied in the proof of Theorem \ref{igusa-thm}.

In {\bf Section \ref{proof-sec}} we conclude the proof of Igusa's theorem. After showing (in \S \ref{weil-subsec} and \S \ref{split-multiplicative-subsec}) that we can actually reduce to the case where $E$ has (split) multiplicative reduction at the prime $\infty$ of $F$, we review basic facts on Tate curves in \S \ref{tate-subsec}, and then give in \S \ref{horizontal-subsec} and \S \ref{vertical-subsec} crucial results on the ``horizontal'' and ``vertical'' variation of the Galois groups (this suggestive terminology is borrowed from Lang's book \cite{la}). In particular, in Proposition \ref{hor-control-pro} we prove that $\gal{F(E[\ell])}{F}=\Gamma_\ell$ for almost all primes $\ell\not=p$, and this result can be usefully applied to gain control on the Galois cohomology of elliptic curves (cf., e.g., \cite{bro}, \cite{vi}). Finally, with all the ``geometric'' results at our disposal, in \S \ref{conclusion-subsec} we finish the proof of Theorem \ref{igusa-thm}. It should be noted that, as in \cite[Ch. 17, \S 5]{la} and \cite[Ch. IV, \S 3.4]{se2}, the final steps in the proof are of a purely algebraic nature: they are just a formal ``juggling'' in abstract group theory and have really nothing to do with elliptic curves.

The subsequent part, {\bf Section \ref{applications-sec}}, is devoted to an interesting arithmetic consequence of Igusa's theorem: we show (Theorem \ref{main-application-thm}) that on a non-isotrivial elliptic curve $E_{/F}$ there are only finitely many torsion points rational over abelian extensions of $F$. Results in the same spirit have been applied in various arithmetic contexts (see, e.g., \cite{bl}, \cite{bre}, \cite{bro}, \cite{vi}), but it seems that the statement above was never explicitly proved in this form.

The article is closed by {\bf Appendix A}, which deals with isotrivial elliptic curves over $F$. These are precisely the elliptic curves over $F$ having a ring of endomorphisms which is larger than $\Z$, and we show that in this situation Theorem \ref{igusa2-thm} is always false. In this case the image of Galois is not open in $\hat\Gamma$, hence in particular it cannot be ``as large as possible'' (in analogy to what happens with CM elliptic curves over number fields).\\

To conclude this introduction, we would like to spend a few words on the background required of the reader. In order to make this note reasonably self-contained, we have tried to keep the prerequisites to a minimum. In fact, apart from basic results in Galois theory and algebraic number theory, we only assume a knowledge of the first definitions and properties in the arithmetic of elliptic curves over local and global fields as treated, for example, in Chapters VII and VIII of Silverman's book \cite{si1}. In particular, we have made an effort not to rely on results in the theory of Lie algebras and Lie groups, contrary to what done in \cite{se2} over number fields. As a consequence, we think that our exposition is more elementary and down-to-earth than those in \cite{la} and \cite{se2}. Moreover, when we introduce more specific notions (e.g., Faltings heights of elliptic curves, Tate's theory of analytic uniformization) we always give complete definitions and suggest references where the interested reader can find details and proofs we have to omit.\\

\noindent\emph{Convention.} Throughout the paper we assume (unless otherwise stated) that $p>3$. This condition is crucially exploited in the proof of
Theorem \ref{shafarevich-thm} (Shafarevich's theorem) to get a uniform upper bound on the degree of the conductor of certain
elliptic curves. We remark, however, that Igusa's results are valid in any positive characteristic.\\

\noindent\emph{Acknowledgements.} We would like to thank Matthias Sch\"utt for useful comments on an earlier version of the paper and Bert Van
Geemen for helpful conversations on some of the topics of this work. We are also grateful to Matthew Baker for pointing us the article \cite{z}
and to Chris Hall for interesting remarks and for showing us alternative proofs of some of the results in this paper. Finally, we thank the
anonymous referee for several valuable remarks and suggestions which led to significant improvements in the exposition.

\section{Faltings heights and a theorem of Shafarevich} \label{faltings-sec}

In this section we want to prove two auxiliary results (Theorem \ref{shafarevich-thm} and Theorem \ref{irreducibility-thm}) that will be crucially employed in the course of our main arguments.

\subsection{Review of Faltings heights of elliptic curves} \label{faltings-heights-subsec}

Before giving precise definitions in the situation we are interested in, let us briefly describe the idea of ``heights'' over global fields in its most basic form. Intuitively speaking, in all its various manifestations the notion of height captures the ``size'' or ``complexity'' of objects of an arithmetic nature. In the simplest case, take a point $P$ in the projective space $\PP^n(\Q)$. Since $\Z$ is a PID, we can find homogeneous coordinates for $P$ of the form
\[ P=[x_0:\dots:x_n] \]
with $x_0,\dots,x_n\in\Z$ and the greatest common divisor of the $x_i$ equal to $1$. Then the \emph{height} of $P$ is naturally defined as
\[ h(P):=\max\bigl\{|x_0|,\dots,|x_n|\bigr\}. \]
Notice that the set
\[ \bigl\{P\in\PP^n(\Q)\mid h(P)\leq C\bigr\} \]
is finite for any constant $C$ (in fact, it has fewer than $(2C+1)^{n+1}$ elements). This sort of finiteness property is one of the most useful features of a well-defined height function (cf. Propositions \ref{finiteness-pro} and \ref{finiteness-height-pro} below). A similar definition can be given over any global field, and this can be usefully applied (at least in principle) to obtain finiteness results for much more complex arithmetic objects. For example, one can embed the (compactified) moduli space of (isomorphism classes of) elliptic curves into a suitable projective space, and then compute the height of an elliptic curve over a global field by means of the chosen embedding. As a consequence, for every $C$ there will be only finitely many (isomorphism classes of) elliptic curves whose height is bounded by $C$. Actually, this idea can be effectively exploited without going through all the geometry which underlies the above considerations. This is achieved by making \emph{a posteriori} all the definitions explicit and then proving that one has a height function on the objects of interest which enjoys the desired formal properties. As will become apparent below, this will be the course taken in our article.\\

After this brief panoramic detour, let us return to our characteristic $p$ setting. Retain the previous notation; in particular, $F$ is the function field of $\cC_{/\mathbb F_r}$ and $A$ is the subring of functions in $F$ that are regular outside the closed point $\infty$. Finally, let $\Sigma_F$ be the set of places of the global field $F$, and if $\p\in\Sigma_F$ denote by $v_{\mathfrak p}$ the discrete valuation associated with $\p$. Note that the elements of $\Sigma_F$ correspond to the (closed) points of $\cC$.

For any $x\in F^\times$ define the principal divisor of $x$ as
\[ (x):=\sum_{\mathfrak p\in\Sigma_F}v_{\mathfrak p}(x)\cdot\mathfrak p. \]
Recall that the degree of a prime $\p\in\Sigma_F$ is by definition the degree over $\F_r$ of the residue field of $\p$, and by linearity the degree of any divisor of $F$ can be introduced; it can be checked that the degree of $(x)$ is $0$ (see, e.g., \cite[Proposition 5.1]{r}). The zero divisor of $x$ is
\[ (x)_0:=\sum_{\mathfrak p\in\Sigma_F}\max\{0,v_{\mathfrak p}(x)\}\cdot\mathfrak p \]
and its pole divisor is
\[ (x)_\infty:=(x^{-1})_0, \]
so we can  write $(x)=(x)_0-(x)_\infty$. We define the \emph{$F$-height} of $x$ to be
\[ h_F(x):=\deg\bigl((x)_0\bigr)=\deg\bigl((x)_\infty\bigr)\in\N. \]
\begin{rem} \label{faltings-rem}
It turns out that $h_F(x)\not=0$ if and only if $x\notin\F_r$ (i.e., if and only if $x$ is not constant), and then
\[ h_F(x)=[F:\F_r(x)]. \]
For a proof of this fact, see \cite[Proposition 5.1]{r}. Observe also that, in the language of algebraic geometry, the integer $h_F(x)$ is the degree of the rational map $\cC_{/\mathbb F_r}\rightarrow\PP^1_{/\mathbb F_r}$ induced by $x$.
\end{rem}
Let now $E_{/F}$ be an elliptic curve and let $\D_E$ be its minimal discriminant divisor; it is the effective divisor of $F$ defined as
\begin{equation} \label{disc-eq}
\D_E:=\sum_{\mathfrak p\in\Sigma_F}v_{\mathfrak p}(\Delta_{\mathfrak p})\cdot\mathfrak p
\end{equation}
where $\Delta_\mathfrak p$ is the discriminant of a minimal Weierstrass equation for $E$ at $\mathfrak p$ in the sense of \cite[Ch. VII, \S 1]{si1} (see also \cite[\S 2]{f1} and \cite[Ch. VIII, \S 8]{si1}). For the purposes of the present paper, we give
\begin{dfn} \label{faltings-dfn}
The \emph{Faltings height} (over $F$) of $E_{/F}$ is the rational number
\[ h_F(E):=\frac{1}{12}\deg(\D_E)\geq0. \]
\end{dfn}
\begin{rem}
Although the notation adopted is the same, the height on $F^\times$ and the Faltings height over $F$ are calculated on objects of a different nature, so no confusion is likely to arise.
\end{rem}
The height $h_F(E)$ is an invariant of the class of $F$-isomorphism of $E_{/F}$. There is another natural notion of height of an elliptic curve $E_{/F}$, essentially equal to the $F$-height of its $j$-invariant $j(E)$, that we recall below.
\begin{dfn} \label{geom-faltings-dfn}
The \emph{geometric Faltings height} (over $F$) of $E_{/F}$ is the rational number
\[ h_{F,g}(E):=\frac{1}{12}h_F\bigl(j(E)\bigr)\geq0. \]
\end{dfn}
The height $h_{F,g}(E)$ is evidently an invariant of the $\bar{F}$-isomorphism class of $E_{/F}$; moreover, by Remark \ref{faltings-rem}:
\[ h_{F,g}(E)>0 \quad \Longleftrightarrow \quad \text{$E$ is not isotrivial.} \]
The following proposition establishes a fundamental relation between $h_F$ and $h_{F,g}$.
\begin{pro} \label{heights-diseq-pro}
The inequality
\[ h_{F,g}(E)\leq h_F(E) \]
holds for every elliptic curve $E_{/F}$.
\end{pro}
\begin{proof} Let $E_{/F}$ be an elliptic curve, and set
\[ T:=\bigl\{\mathfrak p\in\Sigma_F\mid v_{\mathfrak p}(\Delta_\mathfrak p)>0\bigr\}, \qquad T':=\bigl\{\mathfrak p\in\Sigma_F\mid v_{\mathfrak p}(j(E))<0\bigr\} \]
where $\Delta_\mathfrak p$ is the discriminant for $E$ at $\mathfrak p$ introduced in \eqref{disc-eq}. Note that $T'$ consists of the places at which $E$ does not have potential good reduction (cf. \cite[Ch. VII, Proposition 5.5]{si1}). Locally at $\mathfrak p$ we can write $j(E)=c^3_{4,\mathfrak p}\big/\Delta_{\mathfrak p}$ where $c_{4,\mathfrak p}$ is a polynomial expression in the coefficients of a minimal Weierstrass equation for $E$ at $\mathfrak p$ (see \cite[Ch. III, \S 1]{si1} for a precise formula). In particular, $c_{4,\mathfrak p}$ is an integer in the completion $F_{\mathfrak p}$ of $F$ at $\p$. Thus we obtain:
\[ v_{\mathfrak p}(j(E))=3v_{\mathfrak p}(c_{4,\mathfrak p})-v_{\mathfrak p}(\Delta_{\mathfrak p})\geq-v_{\mathfrak p}(\Delta_{\mathfrak p}). \]
We immediately deduce that
\begin{itemize}
\item $T'\subset T$;
\item $-v_{\mathfrak p}(j(E))\leq v_{\mathfrak p}(\Delta_{\mathfrak p})$.
\end{itemize}
Now, by equation \eqref{disc-eq} we can write
\[ \D_E=\sum_{\mathfrak p\in T}v_{\mathfrak p}(\Delta_{\mathfrak p})\cdot\mathfrak p. \]
Hence:
\begin{align*}
h_F\bigl(j(E)\bigr)&=\sum_{\mathfrak p\in T'}\bigl(-v_{\mathfrak p}(j(E))\bigr)\deg(\mathfrak p)\\
                   &\leq\sum_{\mathfrak p\in T'}v_{\mathfrak p}(\Delta_{\mathfrak p})\deg(\mathfrak p)\\
                   &\leq\sum_{\mathfrak p\in T}v_{\mathfrak p}(\Delta_{\mathfrak p})\deg(\mathfrak p)=\deg(\D_E),
\end{align*}
and this proves the proposition by definition of the two Faltings heights. \end{proof}
\begin{rem}
The heights $h_F$ and $h_{F,g}$, introduced for abelian varieties of arbitrary dimension by Faltings in his landmark paper \cite{fa} in which he proved (among others) the Mordell conjecture, admit the simple expressions given in Definitions \ref{faltings-dfn} and \ref{geom-faltings-dfn} because we are working with elliptic curves (i.e., in dimension one) and our base fields are function fields. In particular, in our setting there are no archimedean valuations, so no logarithmic error terms appear in the expression of $h_F(E)$ (cf. \cite[Proposition 1.1]{si2} for a formula in the number field case). The reader may wish to consult \cite{hs} for details and for results related to Lang's conjecture on lower bounds for the N\'eron-Tate canonical height of non-torsion points on an elliptic curve $E$ in terms of the Faltings height $h_F(E)$ (see, in particular, \cite[Theorem 6.1]{hs}).

For a general discussion of the theory of heights of abelian varieties over global fields (albeit with a slant towards the number field setting) we refer the reader to the papers \cite{d} by Deligne and \cite{sz} by Szpiro. In any case, the notion of height of an elliptic curve will play in the sequel only an auxiliary (and limited) role, so in order to keep things as plain as possible we decided to adopt the somewhat \emph{ad hoc} definitions given above.
\end{rem}
\begin{rem}
By normalizing by the so-called ``degree'' of $F$ (which is defined as the smallest value of $h_F(x)$ with $x$ varying in the non-constant elements of $F$, cf. \cite[\S 1]{f2}), it would be possible to modify Definitions \ref{faltings-dfn} and \ref{geom-faltings-dfn} and introduce ``absolute'' versions of the heights of $E$ which do not depend on the field taken as field of definition of $E$. However, since the properties of the functions $h_F$ and $h_{F,g}$ will suffice for our goals, in this note we are content with the ``relative'' (and more common) notions explained above.
\end{rem}
\begin{lem} \label{finiteness-lem}
If $\mathfrak d$ is a divisor of $F$ then the number of the $x\in F^\times$ such that $(x)=\mathfrak d$ is either $0$ or the cardinality of $\F^\times_r$.
\end{lem}
\begin{proof} If the set of such $x$ is not empty, let $x,y\in F^\times$ be such that $(x)=(y)$. Then $(xy^{-1})=0$, and we conclude that $xy^{-1}\in \F^\times_r$ by \cite[Proposition 5.1]{r}. \end{proof}
\begin{pro} \label{finiteness-pro}
The set $\{x\in F^\times\mid h_F(x)\leq C\}$ is finite for all $C\geq0$.
\end{pro}
\begin{proof} Since the function $h_F$ is $\N$-valued on $F^\times$, the claim of the proposition is equivalent to the assertion that the set
\[ \mathscr B_n:=\{x\in F^\times\mid h_F(x)=n\} \]
is finite for all integers $n\geq0$. To begin with, note that if $h_F(x)=n$ then both $(x)_0$ and $(x)_\infty$ are effective divisors of degree $n$. Since $F$ has a finite field of constants, by \cite[Lemma 5.5]{r} the set $\texttt{Div}^+_n$ of effective divisors of degree $n$ is finite. We immediately conclude from Lemma \ref{finiteness-lem} that every fibre of the natural map
\[ \begin{array}{ccc}
   \mathscr B_n & \longrightarrow & \texttt{Div}^+_n\times\texttt{Div}^+_n \\[2mm]
   x & \longmapsto & \bigl((x)_0,(x)_\infty\bigr)
   \end{array} \]
is finite, and the proposition is proved. \end{proof}
As a straightforward consequence of the above proposition, we get
\begin{cor} \label{geom-finiteness-cor}
For all $C\geq0$ there are only finitely many $\bar{F}$-isomorphism classes of elliptic curves $E_{/F}$ with $h_{F,g}(E)\leq C$.
\end{cor}
\begin{proof} Immediate from Proposition \ref{finiteness-pro} by definition of $h_{F,g}(E)$. \end{proof}
We end this $\S$ with a finiteness property of the Faltings height $h_F$ that will be needed in \S \ref{shafarevich-subsec} for the proof of Shafarevich's theorem. Before stating this result, we introduce one more piece of notation: if $S\subset\Sigma_F$ is a finite set we let $A_S$ be the ring of $S$-integers of $F$. In other words:
\[ A_S:=\bigl\{x\in F\mid \text{$v_{\mathfrak p}(x)\geq0$ for all $\p\notin S$}\bigr\}. \]
\begin{pro} \label{finiteness-height-pro}
For all $C\geq0$ there are only finitely many $F$-isomorphism classes of non-isotrivial elliptic curves $E_{/F}$ satisfying $h_F(E)\leq C$.
\end{pro}
See \cite[Corollary 2.5]{si2} for the corresponding statement over number fields.
\begin{proof} Since $h_{F,g}(E)\leq h_F(E)$ by Proposition \ref{heights-diseq-pro}, Corollary \ref{geom-finiteness-cor} implies that there are only finitely many $\bar{F}$-isomorphism classes of (non-isotrivial) elliptic curves $E_{/F}$ with $h_F(E)\leq C$. Thus we are reduced to the following problem: given a non-isotrivial elliptic curve $E_{/F}$, up to $F$-isomorphism there are only finitely many elliptic curves $E'_{/F}$ satisfying
\[ j(E')=j(E), \qquad h_F(E')\leq C. \]
So fix $E$ as above, and note that (by the very definition of the Faltings height!) bounding $h_F(E')$ is equivalent to bounding $\deg(\D_{E'})$.

Now choose a finite set of places $T\subset\Sigma_F$ such that
\begin{itemize}
\item $T$ contains all places of bad reduction for $E$;
\item the ring $A_T$ of $T$-integers is a PID.
\end{itemize}
Recall that $p=\text{char}(F)>3$; then by \cite[Ch. VIII, Proposition 8.7]{si1} we can find an affine Weierstrass equation
\[ E: y^2=x^3+ax+b \]
for $E_{/F}$ with $a,b\in A_T$ and discriminant $\Delta(E):=-16(4a^3+27b^2)\in A^\times_T$ (this result is proved in \emph{loc. cit.} for elliptic curves over number fields, but the arguments work over global function fields too). The elliptic curves over $F$ which are isomorphic over $\bar{F}$ to $E$ are the twists of $E$ given by the equations
\[ E_d: y^2=x^3+ad^2x+bd^3, \qquad d\in F^\times. \]
Furthermore, $E_d\cong E_{d'}$ \emph{over} $F$ if and only if $(d/d')\in(F^\times)^2$. For a proof of these facts, see \cite[Ch. X, Proposition 5.4 and Corollary 5.4.1]{si1}. Our goal is to bound the set of such $d\pmod{(F^\times)^2}$. First of all, without loss of generality we can assume that $d\in A_T$. Moreover, a direct computation shows that
\begin{equation} \label{disc-twist-eq}
\Delta(E_d)=d^6\Delta(E).
\end{equation}
Since $E_d$ is isomorphic to $E$ over a finite extension of $F$ and $E$ has good reduction outside $T$, it follows from \cite[Ch. VII, Proposition 5.4 (b)]{si1} that $E_d$ has either good or additive reduction at primes outside $T$. From \eqref{disc-twist-eq}, if $E_d$ has additive reduction at $\mathfrak p\notin T$ then $v_{\mathfrak p}(d)>0$. More precisely,  if $\mathfrak p\notin T$ then $E_d$ has additive reduction at $\mathfrak p$ if and only if $v_{\mathfrak p}(d)\equiv1\pmod{2}$.\footnote{In fact, if $v_{\mathfrak p}(d)\equiv0\pmod{2}$ then $E_d$ is isomorphic over $F$ to the curve $E_{d'}$ with $d':=d\pi_{\mathfrak p}^{-v_{\mathfrak p}(d)}$, where $\pi_{\mathfrak p}\in A_T$ is such that $v_{\mathfrak p}(\pi_{\mathfrak p})=1$ and $v_{\mathfrak p'}(\pi_{\mathfrak p})=0$ for all primes $\mathfrak p'\notin T$, $\mathfrak p'\not=\mathfrak p$. This shows that $E_d$ has good reduction at $\mathfrak p$.} Hence:
\[ \deg(\D_{E_d})\geq\sum_{\substack{\mathfrak p\notin T\\v_{\mathfrak p}(d)\equiv1\,(2)}}\deg(\p), \]
so a bound on $\deg(\D_{E_d})$ also bounds the degree of those $\p\notin T$ with $v_{\mathfrak p}(d)\equiv1\pmod{2}$. Put $C':=12C$ and define
\[ S:=T\cup\bigl\{\mathfrak p\in\Sigma_F\mid \deg(\p)\leq C'\bigr\}. \]
Since there are only finitely many primes of $F$ of a given degree, the set $S$ is finite. Now observe that the proposition is proved if we show that the set
\begin{equation} \label{quot-set-eq}
\bigl\{d\in F^\times\mid \text{$E_d$ has good reduction outside $S$}\bigr\}\big/(F^\times)^2
\end{equation}
is finite. By the $S$-unit theorem for function fields (see \cite[Proposition 14.2]{r}) we know that the abelian group $A^\times_S/\F^\times_r$ is free of finite rank (equal to $|S|-1$), hence $A^\times_S/(A^\times_S)^2$ is finite. But $A_S$ is a unique factorization domain and $E_d$ has good reduction outside $S$ if and only if $v_{\mathfrak p}(d)\equiv0\pmod{2}$ for all $\mathfrak p\notin S$, so we can assume that $d\in A^\times_S$. Thus the set \eqref{quot-set-eq} injects into $A^\times_S/(A^\times_S)^2$, and we are done. \end{proof}

\subsection{Admissible elliptic curves and the height conjecture} \label{admissible-subsec}

First of all, we introduce a large class of elliptic curves over global function fields of characteristic $p>3$, the so-called ``admissible'' curves: these will be precisely the elliptic curves for which Shafarevich's theorem (Theorem \ref{shafarevich-thm}) holds. Following \cite[Definition 1.2]{f2}, we give
\begin{dfn} \label{admissible-dfn}
An elliptic curve $E_{/F}$ is called \emph{admissible} if the extension $F/\F_r(j(E))$ is finite and separable.
\end{dfn}
Thus an admissible elliptic curve is non-isotrivial. Definition \ref{admissible-dfn} is equivalent to requiring that $j(E)$ is not a $p$th power in $F$.

The next proposition shows that admissibility is preserved under prime-to-$p$ isogenies. This property will be used in the proof of Theorem \ref{irreducibility-thm}.
\begin{pro} \label{admissibility-pro}
Let $E_{/F}$, $E'_{/F}$ be elliptic curves and let $f:E\rightarrow E'$ be an isogeny whose degree is not divisible by $p$. If $E$ is admissible then $E'$ is admissible.
\end{pro}
\begin{proof} For simplicity, write $j:=j(E)$ and $j':=j(E')$. We know that $j$ and $j'$ are linked by an isogeny of degree $d$ not divisible by $p$. Choose an elliptic curve $E_{j'}$ defined over $\F_r(j')$ whose $j$-invariant is equal to $j'$. Then $\F_r(j',E_{j'}[d])$ is separable over $\F_r(j')$ (cf. the references given in the introduction, or adapt the proof of Proposition \ref{sep-torsion-pro} below), and so is $\F_r(j',\Lambda)$ for any subgroup $\Lambda$ of $E_{j'}[d]$. Now $j$ lies in $\F_r(j',\Lambda)$ for some $\Lambda$ as above (take $\Lambda$ equal to the kernel of the corresponding dual isogeny), hence the extension $\F_r(j,j')/\F_r(j')$ is finite and separable. But the extension $F/\F_r(j)$ is finite and separable because $E$ is admissible, hence $F/\F_r(j,j')$ is finite and separable as well. Thus we conclude that $F/\F_r(j')$ is finite and separable, and the proposition is proved. \end{proof}
\begin{rem}
From a highbrow point of view, the separability of the extension $\F_r(j,j')/\F_r(j')$ could also be proved as follows. By decomposing the isogeny $f$ into cyclic (i.e., with cyclic kernel) isogenies we are reduced to the case where $f$ is cyclic of degree $d$. Let $\Phi_d(X,Y)\in\Z[X,Y]$ be the modular polynomial of order $d$, whose definition and main properties can be found, e.g., in \cite[Ch. 5, \S 2]{la}. Then the separability of the above extension can be obtained by exploiting the fact that $\Phi_d(j,j')=0$ and by applying the moduli interpretation explained by Deligne and Rapoport in \cite[Ch. VI, \S 6]{dr}. This separability result is also pointed out by Igusa in the concluding remarks of \cite{i}.
\end{rem}
Now we come to the main result of this $\S$, the so-called ``height conjecture'' for admissible elliptic curves. This conjecture predicts that the Faltings height of an elliptic curve $E_{/F}$ is bounded in terms of the degree of the \emph{conductor} of $E$ and the genus of $F$ (i.e., the genus of the curve $\cC$). The canonical name for this statement is ``height \emph{conjecture}'' because the corresponding assertion in the number field case (and, more generally, for abelian varieties) is still unproved.

Recall that the conductor of an elliptic curve $E$ is the conductor of the Galois representation $T_\ell(E)\otimes\Q_\ell$ (for $\ell\not=p$) as defined in \cite[Ch. IV, \S 10]{si3}. We also refer the reader to \cite[\S 2.1]{se3} for a more conceptual approach to conductors of general $\ell$-adic representations.
\begin{pro}[Height conjecture] \label{height-conj-pro}
Let $E_{/F}$ be an admissible elliptic curve. Then
\begin{equation} \label{height-conj-eq}
h_F(E)\leq\frac{1}{2}\deg(\n_E)+g-1
\end{equation}
where $\n_E$ is the conductor of $E$ over $F$ and $g$ is the genus of $F$.
\end{pro}
\begin{proof} We follow the proof of \cite[Theorem 5.1]{hs} closely. For any prime $\p$ of $F$ we denote $j(E)_{\mathfrak p}$ the reduction of $j(E)$ modulo $\p$ (with $j(E)_{\mathfrak p}=\infty$ if $v_{\mathfrak p}(j(E))<0$) and $e(\p)$ the ramification index of $\p$ over $\F_r(j(E))$.

To begin with, by \cite[Proposition 5.1]{r} there are equalities
\begin{equation} \label{degree-ext-eq}
\begin{split}
\bigl[F:\F_r(j(E))\bigr]&=\sum_{j(E)_{\mathfrak p}=0}v_{\mathfrak p}(j(E))\deg(\p)=\sum_{j(E)_{\mathfrak p}=\infty}-v_{\mathfrak p}(j(E))\deg(\p)\\[2mm]
&=\sum_{j(E)_{\mathfrak p}=1728}v_{\mathfrak p}\bigl(j(E)-1728\bigr)\deg(\p)
\end{split}
\end{equation}
(for the last term note that $\F_r(j(E))=\F_r(j(E)-1728)$). Furthermore, since $j(E)$ and $j(E)-1728$ are primes of $\F_r(j(E))$, the following relations hold:
\begin{equation} \label{ram-indices-eq}
\begin{array}{lcl}
j(E)_{\mathfrak p}=0 & \Longrightarrow & e(\p)=v_{\mathfrak p}(j(E)),\\[2mm]
j(E)_{\mathfrak p}=\infty & \Longrightarrow & e(\p)=-v_{\mathfrak p}(j(E)),\\[2mm]
j(E)_{\mathfrak p}=1728 & \Longrightarrow & e(\p)=v_{\mathfrak p}\bigl(j(E)-1728\bigr).
\end{array}
\end{equation}
Now $E$ is admissible by assumption, so we can apply the Riemann-Hurwitz formula (\cite[Theorem 7.16]{r}) to the  finite separable extension $F/\F_r(j(E))$, and since the genus of $\F_r(j(E))$ is $0$ we get the inequality
\begin{equation} \label{rh-eq}
2g-2\geq-2\bigl[F:\F_r(j(E))\bigr]+\sum_{\mathfrak p\in\Sigma_F}\bigl(e(\p)-1\bigr)\deg(\p).
\end{equation}
Our strategy is to estimate the integers $e(\p)$ and then apply \eqref{rh-eq} to get the desired inequality \eqref{height-conj-eq}. More precisely, the $e(\p)$ for the primes $\p$ with $j(E)_{\mathfrak p}\in\{0,\infty,1728\}$ are computed by means of \eqref{ram-indices-eq}, while for the other primes we simply use the inequality $e(\p)\geq1$. Keeping \eqref{degree-ext-eq} and \eqref{ram-indices-eq} in mind, we can write
\begin{equation} \label{first-eq}
\begin{split}
2\bigl[F:\F_r(j(E))\bigr]&=\frac{5}{6}\sum_{j(E)_{\mathfrak p}=\infty}e(\p)\deg(\p)+\frac{2}{3}\sum_{j(E)_{\mathfrak p}=0}e(\p)\deg(\p)\\[2mm]
&\quad+\frac{1}{2} \sum_{j(E)_{\mathfrak p}=1728}e(\p)\deg(\p).
\end{split}
\end{equation}
Let $\Sigma'_F$ be the set of primes of $F$ such that $j(E)_{\mathfrak p}\notin\{0,\infty,1728\}$ and recall the definition \eqref{disc-eq} of the minimal discriminant divisor $\D_E$. Then, after some easy manipulations, formulas  \eqref{rh-eq} and \eqref{first-eq} yield
\begin{equation} \label{second-eq}
\begin{split}
\frac{1}{6}\deg(\D_E)-2g+2&\leq\sum_{j(E)_{\mathfrak p}=\infty}\bigg(\frac{1}{6}v_{\mathfrak p}(\Delta_{\mathfrak p})-\frac{1}{6}e(\p)+1\bigg)\deg(\p)\\[2mm]
&\quad+\sum_{j(E)_{\mathfrak p}=0}\bigg(\frac{1}{6}v_{\mathfrak p}(\Delta_{\mathfrak p})-\frac{1}{3}e(\p)+1\bigg)\deg(\p)\\[2mm]
&\quad+\sum_{j(E)_{\mathfrak p}=1728}\bigg(\frac{1}{6}v_{\mathfrak p}(\Delta_{\mathfrak p})-\frac{1}{2}e(\p)+1\bigg)\deg(\p)\\[2mm]
&\quad+\sum_{\mathfrak p\in\Sigma'_F}\bigg(\frac{1}{6}v_{\mathfrak p}(\Delta_{\mathfrak p})-e(\p)+1\bigg)\deg(\p).
\end{split}
\end{equation}
At this point, we examine all coefficients according to the different reduction types of $E$. These coefficients can be calculated using \cite[Table 4.1]{si3}, and the results are identical to the ones in \cite[Table 1]{hs}: here we give the details in two of the possible cases just to illustrate the methods. As customary, to denote the reduction type of $E$ at a prime $\p$ of $F$ we use Kodaira symbols, as in the two references given above.
\vskip 2mm
\noindent {\bf 1)} $j(E)_{\mathfrak p}=\infty$ (i.e., $v_{\mathfrak p}(j(E))<0$), reduction type $\text{I}_n^\ast$.\\
\newline Here $v_{\mathfrak p}(\Delta_{\mathfrak p})=n+6$, $e(\p)=-v_{\mathfrak p}(j(E))=n$ and $v_{\mathfrak p}(\n_E)=2$. One then has
\[ \frac{1}{6}v_{\mathfrak p}(\Delta_{\mathfrak p})-\frac{1}{6}e(\p)+1=2=v_{\mathfrak p}(\n_E). \]
\vskip 2mm
\noindent {\bf 2)} $j(E)_{\mathfrak p}=1728$ (i.e., $v_{\mathfrak p}(j(E)-1728)>0$), reduction type $\text{III}$.\\
\newline For an equation $y^2=x^3+ax+b$ of $E$ minimal at $\p$ one has
\[ j(E)-1728=-1728\left(\frac{(4^3a^3)}{\Delta}+1\right)=\frac{-1728}{\Delta}(-16\cdot27b^2), \]
so $e(\p)=v_{\mathfrak p}(j(E)-1728)=2v_{\mathfrak p}(b)-v_{\mathfrak p}(\Delta_{\mathfrak p})$. Under our assumption on the reduction, $v_{\mathfrak p}(\Delta_{\mathfrak p})=3$, $e(\p)=v_{\mathfrak p}(j(E)-1728)=2v_{\mathfrak p}(b)-3\geq1$ and $v_{\mathfrak p}(\n_E)=2$. One then has
\[ \frac{1}{6}v_{\mathfrak p}(\Delta_{\mathfrak p})-\frac{1}{2}e(\p)+1\leq1<v_{\mathfrak p}(\n_E). \]
\vskip 2mm
\noindent All other cases can be dealt with in an analogous manner. The crucial point is that the coefficients of $\deg(\p)$ in the right hand side of \eqref{second-eq} are always bounded from above by $v_{\mathfrak p}(\n_E)$. Hence inequality \eqref{second-eq} gives
\[ \frac{1}{6}\deg(\D_E)-2g+2\leq\sum_{\mathfrak p\in\Sigma_F}v_{\mathfrak p}(\n_E)\deg(\p), \]
and finally
\[ h_F(E)\leq\frac{1}{2}\deg(\n_E)+g-1 \]
by definition of the Faltings height of $E$. \end{proof}
This proposition is clearly interesting in its own right, and we refer the reader to \cite[\S 2]{f1} and \cite[\S 1 (b)]{f2} for comments, consequences and a slightly different proof. For the purposes of this paper its importance lies in the central role it will play in the proof of Shafarevich's theorem: in fact, inequality \eqref{height-conj-eq} is the ingredient that allows us to deduce Theorem \ref{shafarevich-thm} from the finiteness property of the Faltings height $h_F$ that was established in Proposition \ref{finiteness-height-pro}.

\subsection{Shafarevich's theorem} \label{shafarevich-subsec}

A well-known theorem due to Shafarevich (\cite[Ch. IX, Theorem 6.1]{si1}) asserts that if $K$ is a number field there are only a finite number of $K$-isomorphism classes of elliptic curves defined over $K$ having good reduction at all primes of $K$ outside a fixed finite subset. We now show that an analogous result holds for elliptic curves over a global function field $F$ of characteristic $p>3$ which are admissible in the sense of Definition \ref{admissible-dfn}. This seems to be well known to experts and is essentially a consequence of the basic properties of Faltings heights recalled in \S \ref{faltings-heights-subsec}; however, we were not able to track down a reasonably self-contained reference in the literature, so for the convenience of the reader we give a detailed proof of this result.

Before turning to Shafarevich's theorem, let us state the following result, which will be used to prove Corollary \ref{F-isogeny-cor} below.
\begin{pro} \label{isogeny-conductor-pro}
Let $E_{/F}$, $E'_{/F}$ be elliptic curves and let $f:E\rightarrow E'$ be an isogeny defined over $F$. Then $E$ and $E'$ have the same conductor over $F$.
\end{pro}
\begin{proof} It is easy to see that the Tate modules $T_\ell(E)$ and $T_\ell(E')$ are isomorphic as $G_F$-modules for $\ell\not=p$, and the claim follows. \end{proof}
Now we prove
\begin{thm}[Shafarevich] \label{shafarevich-thm}
Let $S$ be a finite set of primes of $F$. There are only finitely many $F$-isomorphism classes of admissible elliptic curves $E_{/F}$ having good reduction at all primes of $F$ outside $S$.
\end{thm}
\begin{proof} Let $S=\{\p_1,\dots,\p_n\}$. If an elliptic curve $E_{/F}$ has good reduction outside $S$ then the support of its conductor $\n_E$ is contained in $S$. Since we are assuming that $p>3$, it follows that\footnote{The uniform bound \eqref{cond-deg-eq} does not hold in characteristic $p=2,3$ due to the possible high divisibility of $\n_E$ by places of additive reduction for $E$.}
\begin{equation} \label{cond-deg-eq}
\deg(\n_E)\leq C(S):=2\sum_{j=1}^n\deg(\p_j).
\end{equation}
Note that the constant $C(S)$ is independent of $E$. Now, by Proposition \ref{height-conj-pro} we know that if $E_{/F}$ is admissible then
\begin{equation} \label{height-bound-eq}
h_F(E)\leq\frac{1}{2}\deg(\n_E)+g-1.
\end{equation}
Combining \eqref{height-bound-eq} and \eqref{cond-deg-eq} yields the inequality
\[ h_F(E)\leq\frac{1}{2}C(S)+g-1, \]
and the claim of the theorem follows from Proposition \ref{finiteness-height-pro}. \end{proof}
It seems worthwhile to point out a straightforward consequence of Shafarevich's theorem.
\begin{cor} \label{F-isogeny-cor}
Every $F$-isogeny class of elliptic curves defined over $F$ contains only finitely many $F$-isomorphism classes of admissible elliptic curves.
\end{cor}
\begin{proof}
By Proposition \ref{isogeny-conductor-pro}, two elliptic curves defined over $F$ which are $F$-isogenous have the same conductor over $F$, and so they have the same set of primes of bad reduction. \end{proof}
The reader is referred to \cite[Ch. IX, Corollary 6.2 and Remark 6.5]{si1} for the counterpart of this result over number fields and for interesting arithmetic consequences of a (still-to-be-found) proof of Shafarevich's theorem over number fields which did not use Siegel's theorem or diophantine approximation techniques.
\begin{rem}
If the word ``admissible'' in Theorem \ref{shafarevich-thm} is replaced by ``non-isotrivial'', the resulting statement is false. For example, let $E_{/F}$ be a non-isotrivial elliptic curve and let $S$ be the set of places of bad reduction for $E$. Moreover, for all $i\geq0$ let $E_i$ be the elliptic curve obtained by applying to $E$ the $i$th iteration of the $r$th-power (relative) Frobenius. Then the family $\{E_i\}_{i\geq0}$ together with the finite set $S$ gives a counterexample to this stronger assertion (cf. Proposition \ref{isogeny-conductor-pro}).
\end{rem}
\begin{rem} \label{shafarevich-rem}
The reader might wonder why we did not mimic, in the proof of Theorem \ref{shafarevich-thm}, the arguments originally given by Tate for elliptic curves over number fields and reproduced, for example, in \cite[Ch. IV, \S 1.4]{se2} and \cite[Ch. IX, Theorem 6.1]{si1}. The reason is simply that they do not work in our function field setting. In fact, the proof by Tate is based on applying Siegel's finiteness theorem for $S$-integral points to a certain auxiliary elliptic curve with $j=0$. Unfortunately, the analogue over global function fields of Siegel's result is valid in full strength only for non-isotrivial elliptic curves (cf. \cite[Lemma 5.1 and Theorem 5.3]{vo}).
\end{rem}

\subsection{Irreducibility results} \label{irreducibility-subsec}

As in \cite[Ch. 17]{la}, the first step towards Igusa's theorem is an irreducibility property for the Tate modules of a non-isotrivial elliptic curve $E$, which we show in Theorem \ref{irreducibility-thm} below. As we shall see, its proof rests upon Shafarevich's theorem.
\begin{lem} \label{end-lem}
Let $E_{/K}$ be an elliptic curve over a field $K$ of positive characteristic $p$. Then $\mathrm{End}(E)=\Z$ if and only if $j(E)\notin\bar{\F}_p\cap K$.
\end{lem}
Here $\F_p$ is the finite field with $p$ elements.
\begin{proof} This is part of a classical result of Deuring, a complete statement and a proof of which can be found in \cite[p. 217]{mu}. \end{proof}
As before, let $F$ be our global function field of characteristic $p>3$. In the following let $E_{/F}$ be a non-isotrivial elliptic curve. Write $V_\ell(E):=T_\ell(E)\otimes_{\mathbb Z_\ell}\Q_\ell$ for all prime numbers $\ell\not=p$.
\begin{thm} \label{irreducibility-thm}
Retain the above notation. Then:
\begin{itemize}
\item[\emph{i)}] the $G_F$-module $E[\ell]$ is irreducible for almost all primes $\ell\not=p$;
\item[\emph{ii)}] the $G_F$-module $V_\ell(E)$ is irreducible for all primes $\ell\not=p$.
\end{itemize}
\end{thm}
\begin{proof} Since $E$ is non-isotrivial, there exists an integer $m\geq0$ such that
\[ j(E)\in F^{p^m},\qquad j(E)\notin F^{p^{m+1}}. \]
In particular, $E$ is isomorphic over a finite, separable extension $L$ of $F$ to (the base change to $F$ of) an elliptic curve $E'$ which is defined and admissible over $F_m:=F^{p^m}$. Now notice that it clearly suffices to show that \emph{i)} and \emph{ii)} hold when $G_F$ is replaced by the smaller absolute Galois group $G_L$ of $L$. On the other hand, the group of automorphisms of a purely inseparable extension is trivial, hence there are natural identifications
\[ \gal{F^s}{F}=\text{Aut}(\bar F/F)=\text{Aut}(\bar F/F_m)=\gal{F_m^s}{F_m}. \]
Thus, up to replacing $F$ by $F_m$ and $E$ by $E'$, in order to prove the theorem it is not restrictive to assume that $E_{/F}$ is admissible, which we do.

\emph{i)} Suppose that $E[\ell]$ is reducible for infinitely many primes $\ell\not=p$, and for any such prime let $H_\ell\subset E[\ell]$ be a nonzero $G_F$-invariant proper subspace. It follows that the elliptic curve $E_\ell:=E/H_\ell$ can be defined over $F$; moreover, the natural (cyclic) isogeny $\pi_\ell:E\twoheadrightarrow E_\ell$ is defined over $F$ and has degree $\ell$. Now we claim that $E_\ell$ and $E_{\ell'}$ are not isomorphic (over $\bar{F}$) if $\ell$ and $\ell'$ are distinct primes (which is equivalent to saying that $j(E_\ell)\not=j(E_{\ell'})$ if $\ell\not=\ell'$). In fact, suppose that
\[ \delta:E_\ell\overset{\cong}{\longrightarrow}E_{\ell'} \]
is an isomorphism, and consider the map $\psi:=\hat{\pi}_{\ell'}\circ\delta\circ\pi_\ell$ where $\hat{\pi}_{\ell'}$ is the dual isogeny to $\pi_{\ell'}$. Since $\text{End}(E)=\Z$ by Lemma \ref{end-lem}, there exists an integer $n$ such that
\begin{equation} \label{psi-eq}
\psi=[n]
\end{equation}
as endomorphisms of $E$. Hence, passing to the degrees on both sides of \eqref{psi-eq}, we get the equality $\ell\ell'=n^2$, which is impossible because $\ell\not=\ell'$. It follows that the set
\[ \E:=\{E_\ell\mid\text{$\ell\not=p$ a prime such that $E[\ell]$ is $G_F$-reducible}\} \]
consists of infinitely many elliptic curves which are defined over $F$, are isogenous to $E$ and are pairwise not $F$-isomorphic. Furthermore, since $E$ is admissible and for every $\ell\not=p$ as above the isogeny $\pi_\ell$ has degree $\ell$, Proposition \ref{admissibility-pro} ensures that all of them are admissible. But every curve in $\E$, being isogenous \emph{over} $F$ to $E$, has good reduction outside the support of the conductor of $E$ (which consists of finitely many primes of $F$), and this contradicts Theorem \ref{shafarevich-thm}.

\emph{ii)} If $X\subset V_\ell(E)$ is a $G_F$-invariant, one-dimensional $\Q_\ell$-vector subspace then $X\cap T_\ell(E)$ is a $G_F$-invariant submodule of $T_\ell(E)$ which is free of rank one over $\Z_\ell$, so the irreducibility result for $V_\ell(E)$ is proved once we show that $T_\ell(E)$ is simple as a $G_F$-module. Thus suppose that this is not the case, and let $W\subset T_\ell(E)$ be a nonzero proper $\Z_\ell$-submodule which is $G_F$-invariant. Then $W$ is free of rank one over $\Z_\ell$, and by the linearity of the $G_F$-action we can assume that it is a direct summand of $T_\ell(E)$. For every $n\geq1$ let $\lambda_n:T_\ell(E)\twoheadrightarrow E[\ell^n]$ be the canonical projection, and set $W_n:=\lambda_n(W)$. It follows that $W_n$ is cyclic of order $\ell^n$ and $G_F$-invariant, and the elliptic curve $E_n:=E/W_n$ is defined over $F$ and isogenous to $E$ over $F$. Moreover, the natural (cyclic) isogeny $E\twoheadrightarrow E_n$ has degree $\ell^n$. We contend that $E_n$ and $E_m$ are non-isomorphic if $n<m$, which can be seen as follows. There is an obvious isogeny $\pi_{n,m}:E_n\twoheadrightarrow E_m$ which is defined over $F$, is separable and has a cyclic kernel (of order $\ell^{m-n}$). Suppose now that
\[ \delta:E_m\overset{\cong}{\longrightarrow}E_n \]
is an isomorphism, and consider the separable map $\psi:=\delta\circ\pi_{n,m}$. Since $E$ is admissible, the curve $E_n$ is admissibile (so, in particular, non-isotrivial) as well by Proposition \ref{admissibility-pro}. Now we know by Lemma \ref{end-lem} that $\text{End}(E_n)=\Z$, hence
\[ \psi=[t] \]
for a certain integer $t$ not divisible by $p$. But $\psi$ has a cyclic kernel (equal to the kernel of $\pi_{n,m}$), while the kernel of $[t]$, being isomorphic to $(\Z/t\Z)^2$, is not cyclic. This proves our claim, and we conclude as before that the set
\[ \E:=\{E_n\mid n\geq1\} \]
consists of infinitely many admissible elliptic curves defined over $F$ and belonging to different $F$-isomorphism classes. Since all the curves in $\E$ have good reduction outside the support of the conductor of $E$, this contradicts once again Theorem \ref{shafarevich-thm}. \end{proof}

\section{Proof of Theorem \ref{igusa-thm}} \label{proof-sec}

Now that we have collected some of the algebraic results that will be used (most notably Theorem \ref{irreducibility-thm}), we can prove Igusa's theorem. This will be done in \S \ref{conclusion-subsec}, and before that we need to gather a handful of geometric features that will pave our way towards Theorem \ref{igusa-thm}.

\subsection{Consequences of the Weil pairing} \label{weil-subsec}

In this $\S$ the curve $E_{/F}$ is an arbitrary elliptic curve, possibly isotrivial. Let $n\geq2$ be an integer not divisible by $p$; recall from \cite[Ch. III, \S 8]{si1} that there is a bilinear, alternating, non-degenerate pairing (called the \emph{Weil pairing})
\[ e_n: E[n]\times E[n] \longrightarrow \boldsymbol\mu_n \]
such that
\begin{equation} \label{weil-galois-eq}
e_n(P^\sigma,Q^\sigma)=e_n(P,Q)^\sigma
\end{equation}
for all $P,Q\in E[n]$ and $\sigma\in\gal{\bar F}{F}$. As a consequence of the properties of $e_n$, it turns out (\cite[Ch. III, Corollary 8.1.1]{si1}) that $\boldsymbol\mu_n\subset F(E[n])$.\\

\noindent\emph{In the rest of the paper we will regard elements of $\gal{F(E[n])}{F}$ as $2\times2$ invertible matrices via the map $\bar\rho_{E,n}$.}
\begin{lem} \label{weil-equality-lem}
The equality
\[ e_n(Q_1,Q_2)^\sigma=e_n(Q_1,Q_2)^{\det(\sigma)} \]
holds for all $Q_1,Q_2\in E[n]$ and all $\sigma\in\gal{F(E[n])}{F}$.
\end{lem}
\begin{proof} Let $P_1,P_2$ be a $\Z/n\Z$-basis for $E[n]$ and let $\bigl(\begin{smallmatrix}a&b\\c&d\end{smallmatrix}\bigr)$ be the matrix of $\sigma$ for this basis. Then, by the properties of the Weil pairing (in particular, the Galois equivariance of $e_n$ expressed by \eqref{weil-galois-eq}), one gets:
\begin{align*}
e_n(P_1,P_2)^\sigma&=e_n(P^\sigma_1,P^\sigma_2)\\
                   &=e_n(aP_1+cP_2,bP_1+dP_2)\\
                   &=e_n(P_1,P_2)^{ad-bc}\\
                   &=e_n(P_1,P_2)^{\det(\sigma)},
\end{align*}
and the equality for arbitrary points $Q_1,Q_2$ in $E[n]$ follows by bilinearity. \end{proof}
The following important results are consequences of the above lemma.
\begin{pro} \label{inclusion-gamma-pro}
For all $n\geq1$ prime to $p$ there is an inclusion $\gal{F(E[n])}{F}\subset\Gamma_n$.
\end{pro}
\begin{proof} The case $n=1$ being trivial, we assume $n\geq2$. Let $\sigma\in\gal{F(E[n])}{F}$; by definition of $\Gamma_n$, we have to show that $\det(\sigma)$ belongs to the subgroup $H_n$ of $(\Z/n\Z)^\times$ generated by $r$.

Choose $P,Q\in E[n]$ such that $\zeta_n:=e_n(P,Q)$ is a primitive $n$th root of unity: this can be done by the non-degeneracy of the Weil pairing. Then
\begin{equation} \label{sigma-action-eq1}
\zeta^\sigma_n=\zeta^{\det(\sigma)}_n
\end{equation}
by Lemma \ref{weil-equality-lem}. On the other hand, there is a natural identification $\gal{\F_r(\boldsymbol\mu_n)}{\F_r}=H_n$, so there exists a positive integer $s$ such that
\begin{equation} \label{sigma-action-eq2}
\zeta^\sigma_n=\zeta^{r^s}_n.
\end{equation}
The claim follows immediately by comparing \eqref{sigma-action-eq1} and \eqref{sigma-action-eq2}. \end{proof}
As in the introduction, denote
\[ \rho_E: G_F \longrightarrow \prod_{\ell\not=p}GL_2(\Z_\ell) \]
the Galois representation attached to $E$ and write $\hat\Gamma$ for the profinite group defined in \eqref{gamma-hat-eq}.
\begin{cor} \label{inclusion-gamma-cor}
There is an inclusion $\rho_E(G_F)\subset\hat\Gamma$.
\end{cor}
\begin{proof} Pass to the projective limit in the inclusions of Proposition \ref{inclusion-gamma-pro}. \end{proof}

\subsection{Reduction to the split multiplicative case} \label{split-multiplicative-subsec}

From here on $E_{/F}$ is non-isotrivial. We explain why it is not restrictive, for the purposes of our paper, to assume that $E_{/F}$ has split multiplicative reduction at $\infty$.
\begin{pro} \label{split-reduction-pro}
There exist a finite separable extension $F'$ of $F$ and a prime $\infty'$ of $F'$ such that the base-changed elliptic curve $E_{/F'}:=E\times_F F'$ has split multiplicative reduction at $\infty'$.
\end{pro}
\begin{proof} By assumption, $j(E)\notin\F_r$. In other words, $j(E)$ is a non-constant function on the smooth projective curve $\cC$, hence there exists a closed point $\infty$ of $\cC$ at which $j(E)$ has a pole. This means that $v_\infty(j(E))<0$, so $E$ has potential multiplicative reduction at $\infty$ (cf. \cite[Ch. VII, Propositions 5.4 and 5.5]{si1}). Thus we can find a finite extension $\tilde{F}$ of $F$ and a prime $\tilde\infty$ of $\tilde F$ above $\infty$ such that $E_{/\tilde F}$ has multiplicative reduction at $\tilde\infty$. At the cost of passing to a quadratic extension $F'$ of $\tilde F$ in which $\tilde\infty$ is inert (equal to a prime $\infty'$), $E$ acquires \emph{split} reduction. Finally, note (see \emph{loc. cit.}) that we are making an extension $\tilde{F}/F$ of degree dividing $24$ followed by an extension $F'/\tilde{F}$ of degree at most two: the separability of $F'/F$ is granted by our assumption that $p>3$. \end{proof}
Now let $L$ be a (not necessarily finite) separable extension of $F$ and let $G_L\subset G_F$ be the corresponding absolute Galois group. In the sequel, Proposition \ref{split-reduction-pro} will be applied in conjunction with the following result.
\begin{pro} \label{restriction-pro}
If $\rho_E(G_L)$ is an open subgroup of $\hat\Gamma$ such is $\rho_E(G_F)$, hence Theorem \ref{igusa-thm} holds for $E_{/F}$ if it holds for $E_{/L}$.
\end{pro}
\begin{proof} This is a simple argument about topological groups. Suppose that $G$ is a topological group and let $H\subset H'$ be subgroups of $G$ with $H$ open in $G$. Then
\[ H'=\bigcup_{h'\in H'}h'H \]
is open in $G$ as well because it is the union of the open subsets $h'H$. To prove the proposition, observe that by Corollary \ref{inclusion-gamma-cor} we already know that $\rho_E(G_F)$ is a subgroup of $\hat\Gamma$, and then apply the above result to the subgroups $\rho_E(G_L)\subset\rho_E(G_F)$ of $\hat\Gamma$. \end{proof}

\subsection{Tate curves: an overview} \label{tate-subsec}

General references for the theory of Tate's analytic uniformization of elliptic curves are \cite[Ch. 15]{la}, \cite[Ch. IV, Appendix A.1]{se2} and \cite[Ch. V]{si3}, and we refer to them for more details and for proofs of the cited results.

Quite generally, in this $\S$ we let $K$ denote a field which is complete with respect to a discrete valuation $v$; we assume that the residue field of $K$ is perfect of characteristic $p>0$. Let $q\in K^\times$ be such that $v(q)>0$, and let $\langle q\rangle$ be the discrete subgroup of $K^\times$ generated by $q$. The \emph{Tate elliptic curve} (relative to $q$) is the curve with Weierstrass equation
\[ E_q:y^2+xy=x^3+a_4(q)x+a_6(q) \]
whose coefficients are given by the power series
\[ a_4(q):=-5\sum_{n\geq1}\frac{n^3q^n}{1-q^n}, \qquad a_6(q):=-\frac{1}{12}\sum_{n\geq1}\frac{(7n^5+5n^3)q^n}{1-q^n}. \]
Since $v(q)>0$, these series converge in the $v$-adic metric. The discriminant and the $j$-invariant of $E_q$ are given by the formulas
\[ \Delta(q)=q\prod_{n\geq1}(1-q^n)^{24}, \qquad j(q)=\frac{1}{q}+744+196884q+\dots, \]
which are clearly reminiscent of the corresponding ones from the complex case. If we define the series
\begin{align*}
x(u,q)&:=\sum_{n\in\mathbb Z}\frac{q^nu}{(1-q^nu)^2}-2\sum_{n\geq1}\frac{nq^n}{1-q^n},\\[2mm]
y(u,q)&:=\sum_{n\in\mathbb Z}\frac{q^{2n}u^2}{(1-q^nu)^3}+\sum_{n\geq1}\frac{nq^n}{1-q^n}
\end{align*}
we obtain a $v$-adic analytic uniformization
\begin{equation} \label{tate-unif-eq}
\begin{array}{rccl}
\phi: & \bar{K}^\times\big/\langle q\rangle & \overset{\cong}{\longrightarrow} & E_q(\bar K)\\[2mm]
      & u & \longmapsto & \bigl(x(u,q),y(u,q)\bigr).
\end{array}
\end{equation}
Since the action of $G_K:=\gal{\bar K}{K}$ on $\bar K$ is $v$-adically continuous, the map $\phi$ defined in \eqref{tate-unif-eq} is $G_K$-equivariant, i.e. $\phi$ is not only an isomorphism of $v$-adic Lie groups but also an isomorphism of $G_K$-modules. Of course, this property is of the utmost importance for arithmetic applications.

Because the $j$-invariant $j(q)$ of $E_q$ is not integral (i.e., $v(j(q))<0$), it is clear that, unlike what happens for elliptic curves over the complex numbers, \emph{not} every elliptic curve over $K$ is analytically isomorphic to a quotient $\bar K^\times/\langle q\rangle$ for some $q\in K^\times$ with $v(q)>0$. More precisely, the reduction $\tilde{E}_q$ of $E_q$ modulo $v$ has the equation
\[ \tilde{E}_q: y^2+xy=x^3, \]
so $E_q$ has split multiplicative reduction over $K$. The crucial point in Tate's theory is that the non-integrality of the $j$-invariant is a necessary and sufficient condition for an elliptic curve $E$ over $K$ to be analytically uniformized as above. Indeed, the following fundamental result holds.
\begin{thm}[Tate] \label{tate-thm}
Let $K$ be as before.
\begin{itemize}
\item[\emph{i)}] For every $q\in K^\times$ with $v(q)>0$ the map
\[ \phi:\bar{K}^\times\big/\langle q\rangle\longrightarrow E_q(\bar K) \]
described in \eqref{tate-unif-eq} is an isomorphism of $G_K$-modules.
\item[\emph{ii)}] For every $j_0\in K^\times$ with $v(j_0)<0$ there is a unique $q\in K^\times$ with $v(q)>0$ such that the Tate elliptic curve $E_{q/K}$ has $j$-invariant $j_0$. The curve $E_q$ is characterized by the equality $j(E_q)=j_0$ and the fact that it has split multiplicative reduction over $K$.
\item[\emph{iii)}] Let $E_{/K}$ be an elliptic curve with non-integral $j$-invariant $j_0\in K^\times$, and let $E_q$ be the Tate curve with $j$-invariant $j_0$ as in $ii)$. If $E$ has split multiplicative reduction then $E$ is isomorphic to $E_q$ over $K$, while if $E$ does not have split multiplicative reduction  then there is a unique quadratic extension $L$ of $K$ such that $E$ is isomorphic to $E_q$ over $L$.
\end{itemize}
\end{thm}
A complete proof of this theorem can be found in \cite[Ch. V]{si3}. If $E_{/K}$ is an elliptic curve with non-integral $j$-invariant, the element $q\in K^\times$ whose existence is established in Theorem \ref{tate-thm} is called ``Tate period'' for $E$.
\begin{cor} \label{tate-representation-cor}
Let $E_{/K}$ be an elliptic curve with non-integral $j$-invariant and split multiplicative reduction, and retain the notation of Theorem \ref{tate-thm}.
\begin{itemize}
\item[\emph{i)}] Let $n\geq0$ be an integer not divisible by $p$, let $\zeta_n$ be a primitive $n$th root of unity in $\bar K$ and fix an $n$th root $q^{1/n}$ of $q$ in $\bar K$. There is an isomorphism
\[ E[n]\cong\big\langle q^{1/n},\zeta_n\big\rangle\big/\langle q\rangle \]
of $G_K$-modules.
\item[\emph{ii)}] Let $\pi$ be a uniformizer of $K$, write $q=\pi^eu$ with $e:=v(q)>0$ and denote $v_\ell$ the $\ell$-adic valuation on $\Q$ for a prime $\ell$. For all primes $\ell\not=p$ and integers $m>v_\ell(e)$ the field $K\bigl(q^{1/\ell^m},\zeta_{\ell^m}\bigr)$ admits an automorphism $\sigma$ over $K$ leaving $\zeta_{\ell^m}$ fixed and such that $\sigma(q^{1/\ell^m})=(\zeta_{\ell^m})^{\ell^{v_\ell(e)}}\cdot q^{1/\ell^m}$. Thus there exists an element
\[ \sigma\in\gal{K(E[\ell^m])}{K} \]
which is represented by $\bigl(\begin{smallmatrix}1&\ell^{v_\ell(e)}\\0&1\end{smallmatrix}\bigr)$ with respect to the basis of $E[\ell^m]$ corresponding to the basis $\bigl\{\zeta_{\ell^m},q^{1/\ell^m}\bigr\}$ of $E_q[\ell^m]$.
\item[\emph{iii)}] With notation and conventions as before, the group $\gal{K(E[\ell^\infty])}{K}$ contains the subgroup
\[ \begin{pmatrix}1&\ell^{v_\ell(e)}\Z_\ell\\0&1\end{pmatrix} \]
of $GL_2(\Z_\ell)$ for all primes $\ell\not=p$.
\end{itemize}
\end{cor}
\begin{proof} Part \emph{i)} is an immediate consequence of part \emph{i)} of Theorem \ref{tate-thm}, while \emph{ii)} follows from Kummer theory. Finally, \emph{iii)} is implied by \emph{ii)}. \end{proof}
We shall use these results in the following situation. By Proposition \ref{restriction-pro}, in order to prove Theorem \ref{igusa-thm} we can extend the ground field $F$ to any separable extension; on the other hand, Proposition \ref{split-reduction-pro} guarantees the existence of a prime of split multiplicative reduction for $E$ in a suitable finite separable extension of $F$. When combined together, these two results say that it is not restrictive for us to assume that the prime $\infty$ of $F$ that we chose at the outset is of split multiplicative reduction for $E$, which from here on we do without any further comment. In particular, there are a Tate period $q\in F^\times_\infty$ and a $\gal{\bar F_\infty}{F_\infty}$-equivariant short exact sequence
\[ 0\longrightarrow\langle q\rangle\longrightarrow\bar F^\times_\infty\longrightarrow E(\bar F_\infty)\longrightarrow0 \]
which expresses the geometric points of $E_{/F_\infty}$ as a quotient of a one-dimensional $\infty$-adic torus by an infinite cyclic subgroup.\\

As an application of Tate's uniformization, we conclude this $\S$ by showing that the prime-to-$p$ torsion of $E$ is rational over $F^s$.
\begin{pro} \label{sep-torsion-pro}
With notation as above, let $k$ be the smallest positive integer such that $q\notin(F^\times_\infty)^{p^k}$ and let $P\in E_\mathrm{tors}(\bar F)$. Then $P\in E_\mathrm{tors}(F^s)$ if and only if $p^k$ does not divide the order of $P$.
\end{pro}
\begin{proof} Let $n$ be the order of $P$ and let $q^{1/n}$ be an $n$th root of $q$ in $\bar F_\infty$. The geometric points in $E[n]$ are rational over $F_\infty(\boldsymbol\mu_n,q^{1/n})\cap\bar F$, and the proposition follows from the next lemma. \end{proof}
\begin{lem}
$F^s=\bar F\cap F^s_\infty$.
\end{lem}
\begin{proof} It is enough to observe that the purely inseparable extensions of $F$ are totally ramified at all places (as can be deduced, e.g., from \cite[Proposition 7.5]{r}). \end{proof}

\subsection{Galois groups: horizontal control} \label{horizontal-subsec}

In this $\S$ we prove two results describing (in a strong way) the asymptotic behaviour of the Galois groups of $F(E[\ell])$ and of $F(E[\ell^\infty])$ over $F$ when $\ell$ varies. As we shall see, the fact that $E$ admits, locally at $\infty$, an analytic uniformization will be crucially exploited.

As a notational convention, if $R$ is a domain denote $Q(R)$ the quotient field of $R$. We begin with two algebraic lemmas.
\begin{lem} \label{H-open-lem}
Let $R$ be either a discrete valuation ring or a (topological) field. Let $H$ be a subgroup of $GL_2(R)$ that acts irreducibly on $Q(R)^2$ and suppose that $H$ contains the subgroup $\bigl(\begin{smallmatrix}1&I\\0&1\end{smallmatrix}\bigr)$ where $I$ is a nonzero ideal of $R$. Then $H$ contains an open subgroup of $SL_2(R)$.
\end{lem}
In particular, if $R$ is a field then $I=R$ and $H=SL_2(R)$.
\begin{proof}[Sketch of proof.] By adapting the proof of \cite[Ch. XIII, Lemma 8.1]{la2} it can be shown that
\[ \Big\langle\begin{pmatrix}1&0\\I&1\end{pmatrix},\begin{pmatrix}1&I\\0&1\end{pmatrix}\Big\rangle\supset\ker\Big(SL_2(R)\longrightarrow SL_2\bigl(R/I^2\bigr)\Big). \]
The kernel on the right is understood to be the whole $SL_2(R)$ if $I=R$. Observe that this kernel is open in $SL_2(R)$ because all nonzero ideals in $R$ are open, hence the quotient $R/I^2$ is discrete. Finally, from the irreducibility condition and the fact that $\bigl(\begin{smallmatrix}1&I\\0&1\end{smallmatrix}\bigr)$ is contained in $H$ one can deduce the existence of a suitable basis of $Q(R)^2$ such that $H$ contains $\bigl(\begin{smallmatrix}1&0\\I&1\end{smallmatrix}\bigr)$ as well (see \cite[Ch. IV, \S 3.2, Lemma 2]{se2} for details), and this completes the proof of the lemma. \end{proof}
\begin{lem} \label{lin-disj-lem}
If $p\nmid n$ then $\gal{F(\boldsymbol\mu_n)}{F}=H_n$.
\end{lem}
\begin{proof} Since $\F_r$ is algebraically closed in $F$ it follows that
\[ F\cap\F_r(\boldsymbol\mu_n)=\F_r, \]
i.e. $F$ and $\F_r(\boldsymbol\mu_n)$ are linearly disjoint over $\F_r$. Thus
\[ \gal{F(\boldsymbol\mu_n)}{F}=\gal{\F_r(\boldsymbol\mu_n)}{\F_r}, \]
whence the claim. \end{proof}
In the sequel let $\F_\ell$ be the field with $\ell$ elements. Now we can prove
\begin{pro} \label{hor-control-pro}
The equality $\gal{F(E[\ell])}{F}=\Gamma_\ell$ holds for almost all primes $\ell\not=p$.
\end{pro}
\begin{proof} First we show that $\gal{F(E[\ell])}{F}$ contains $SL_2(\F_\ell)$ for almost all primes $\ell\not=p$. To begin with, there is a natural embedding
\[ \gal{F_\infty(E[\ell])}{F_\infty}\;\longmono\;\gal{F(E[\ell])}{F} \]
which we interpret as an inclusion, so that we view the former group as a subgroup of the latter. As in \S \ref{tate-subsec}, write $q$ for the Tate period of $E$ at $\infty$, so $E(\bar{F}_\infty)\cong\bar{F}^\times_\infty/\langle q\rangle$ as Galois modules. Now let $\ell$ be a prime different from $p$ not dividing $e=v_\infty(q)$. By part \emph{ii)} of Corollary \ref{tate-representation-cor}, there exists $\sigma\in\gal{F_\infty(E[\ell])}{F_\infty}$ which is represented by the matrix $\bigl(\begin{smallmatrix}1&1\\0&1\end{smallmatrix}\bigr)$ with respect to a suitable $\F_\ell$-basis of $E[\ell]$. But part \emph{i)} of Theorem \ref{irreducibility-thm} says that $E[\ell]$ is an irreducible $\gal{F(E[\ell])}{F}$-module for almost all primes $\ell\not=p$, so our claim follows from Lemma \ref{H-open-lem}.

Now let $\ell\not=p$ be a prime such that $\gal{F(E[\ell])}{F}$ contains $SL_2(\F_\ell)$. As noticed before, one knows that $\boldsymbol\mu_\ell\subset F(E[\ell])$ and that the Galois action on the roots of unity is given by the determinant (Lemma \ref{weil-equality-lem}), so Lemma \ref{lin-disj-lem} ensures that $\gal{F(E[\ell])}{F}$ fits into a short exact sequence
\[ 0\longrightarrow SL_2(\F_\ell)\longrightarrow\gal{F(E[\ell])}{F} \xrightarrow{\text{det}}H_\ell\longrightarrow0. \]
By definition of $\Gamma_\ell$, the proposition is proved. \end{proof}
\begin{rem}
In their paper \cite{ch}, Cojocaru and Hall give a uniform version of Proposition \ref{hor-control-pro}. More precisely, they show that there exists a positive constant $c(F)$, depending at most on the genus of $\cC$, such that $\gal{F(E[\ell])}{F}=\Gamma_\ell$ for any non-isotrivial elliptic curve $E_{/F}$ and any prime number $\ell\geq c(F)$, $\ell\not=p$. Moreover, they determine an explicit expression for $c(F)$: see \cite[Theorem 1]{ch}.
\end{rem}
Now we state an auxiliary result that will be applied in various occasions later on. By defining it componentwise in the obvious manner, consider the determinant map
\[ \det:\hat\Gamma\longrightarrow\hat{H}. \]
This map is the one that appears in \eqref{gamma-hat-eq}. By a slight abuse of notation, we denote in the same way both the determinant map on  $\hat\Gamma$ and the analogous maps on the $\hat\Gamma_\ell$.
\begin{lem} \label{surj-det-lem}
The following hold:
\begin{itemize}
\item[\emph{i)}] $\det(\rho_{E,\ell}(G_F))=\hat{H}_\ell$ for all primes $\ell\not=p$;
\item[\emph{ii)}] $\det(\rho_E(G_F))=\hat{H}$.
\end{itemize}
\end{lem}
\begin{proof} \emph{i)} If $\ell$ is a prime different from $p$, it is an immediate consequence of Lemma \ref{weil-equality-lem} (cf. also \cite[Ch. I, \S 1.2]{se2}) that
\[ \det(\rho_{E,\ell}):G_F\longrightarrow\Z^\times_\ell \]
coincides with the cyclotomic character giving the action of $G_F$ on the $\ell^\infty$th roots of unity. It follows that $\det(\rho_{E,\ell}(G_F))$ identifies with the Galois group $\gal{F(\boldsymbol\mu_{\ell^\infty})}{F}$ where $F(\boldsymbol\mu_{\ell^\infty})$ is the extension of $F$ generated by all roots of unity of order a power of $\ell$. On the other hand, by setting $n=\ell^m$ and passing to the projective limit over $m$ in Lemma \ref{lin-disj-lem} we get that
\[ \gal{F(\boldsymbol\mu_{\ell^\infty})}{F}=\hat{H}_\ell, \]
whence our claim.

Part \emph{ii)} can be proved in exactly the same way, this time working with all roots of unity of prime-to-$p$ order. \end{proof}
\begin{rem}
Lemma \ref{surj-det-lem} is valid for all elliptic curves $E_{/F}$, including isotrivial ones.
\end{rem}
\begin{pro} \label{hor-control-pro2}
For all primes $\ell\not=p$ the group $\gal{F(E[\ell^\infty])}{F}$ is open in $\hat\Gamma_\ell$.
\end{pro}
\begin{proof} Let $\ell\not=p$ be a prime. From the theory of Tate's uniformization (see part \emph{iii)} of Corollary \ref{tate-representation-cor}) we know that
\[ \begin{pmatrix}1&\ell^n\Z_\ell\\0&1\end{pmatrix}\subset\rho_{E,\ell}(G_F)=\gal{F(E[\ell^\infty])}{F}\subset GL_2(\Z_\ell) \]
for $n=v_\ell(e)$ and $e=-v_\infty(j(E))$. But $V_\ell(E)$ is an irreducible $\rho_{E,\ell}(G_F)$-module by part \emph{ii)} of Theorem \ref{irreducibility-thm}, hence $\rho_{E,\ell}(G_F)$ contains an open subgroup of $SL_2(\Z_\ell)$ by Lemma \ref{H-open-lem}. To prove the proposition one can proceed as follows. As a consequence of part \emph{i)} of Lemma \ref{surj-det-lem}, there is a commutative diagram of short exact sequences
\begin{equation} \label{sequences-diag-eq1}
\xymatrix@R=20pt{0\ar[r] & W \ar[r]\ar@{^{(}->}[d] & \rho_{E,\ell}(G_F) \ar[r]^-{\det}\ar@{^{(}->}[d] & \hat{H}_\ell \ar[r]\ar@{=}[d] & 0\\
          0\ar[r] & SL_2(\Z_\ell) \ar[r] & \hat\Gamma_\ell \ar[r]^-{\det} & \hat{H}_\ell \ar[r] & 0}
\end{equation}
where $W:=\rho_{E,\ell}(G_F)\cap\ker(\det)$ and the vertical maps are inclusions. Since $\ker(\det)=SL_2(\Z_\ell)$ and $W$ is an open subgroup of $SL_2(\Z_\ell)$, it follows that $W$ is of finite index in $SL_2(\Z_\ell)$ because this matrix group is compact. But then the exact sequence between the cokernels of the vertical maps in \eqref{sequences-diag-eq1} shows that $\rho_{E,\ell}(G_F)$ has finite index in $\hat\Gamma_\ell$, and this suffices to prove our claim because $\rho_{E,\ell}(G_F)$ is closed in $\hat\Gamma_\ell$. \end{proof}
We conclude this $\S$ with the following
\begin{pro} \label{open-pro}
Let $S$ be a finite set of primes not containing $p$ and let
\[ \hat\Gamma_S:=\prod_{\ell\in S}\hat\Gamma_\ell. \]
Moreover, denote $E[S^\infty]$ the group of points of $E$ of order divisible only by primes in $S$. Then $\gal{F(E[S^\infty])}{F}$ is open in $\hat\Gamma_S$.
\end{pro}
\begin{proof} If $L/F$ is a finite (separable) extension then $\gal{L(E[S^\infty])}{L}$ canonically identifies with the Galois group of $F(E[S^\infty])$ over $F(E[S^\infty])\cap L$, so it can naturally be viewed as an open subgroup of $\gal{F(E[S^\infty])}{F}$. Let $m$ be the product of the primes in $S$ and set $L:=F(E[m])$. Then, as $\ell$ varies in $S$, the extensions $L(E[\ell^\infty])/L$ are pro-$\ell$, thus the fields $L(E[\ell^\infty])$ are pairwise linearly disjoint over $L$. It follows that
\begin{equation} \label{gal-prod-eq}
\gal{L(E[S^\infty])}{L}=\prod_{\ell\in S}\gal{L(E[\ell^\infty])}{L}.
\end{equation}
But the same reasoning as above shows that, for all $\ell$, $\gal{L(E[\ell^\infty])}{L}$ is an open subgroup of $\gal{F(E[\ell^\infty])}{F}$, and then Proposition \ref{hor-control-pro2} implies that $\gal{L(E[\ell^\infty])}{L}$ is open in $\hat\Gamma_\ell$. We readily deduce from \eqref{gal-prod-eq} that $\gal{L(E[S^\infty])}{L}$ is open in $\hat\Gamma_S$, hence $\gal{F(E[S^\infty])}{F}$ is open in $\hat\Gamma_S$ as well. \end{proof}

\subsection{Galois groups: vertical control} \label{vertical-subsec}

In this short $\S$ we take a closer look at the Galois group of $F(E[\ell^\infty])$ over $F$ for a prime number $\ell\not=p$. In order to do this, we need to introduce an algebraic notion that will prove extremely useful.

Let $X$ be a profinite group and let $\Sigma$ be a finite simple group. Following \cite[Ch. IV, \S 3.4]{se2}, we say that $\Sigma$ \emph{occurs} in $X$ if there exist closed subgroups $X_1$, $X_2$ of $X$ such that $X_2$ is normal in $X_1$ and $X_1/X_2\cong\Sigma$.
\begin{lem} \label{occurrence-lem}
With $X$ and $\Sigma$ as above, the following hold:
\begin{itemize}
\item[\emph{i)}] let $Y$ be a closed normal subgroup of $X$; if $\Sigma$ occurs in $X$ then $\Sigma$ occurs in either $Y$ or $X/Y$;
\item[\emph{ii)}] if $X=\varprojlim X/\Omega_\alpha$ with $\Omega_\alpha$ open in $X$ for all $\alpha$ then $\Sigma$ occurs in $X$ if and only if $\Sigma$ occurs in $X/\Omega_\alpha$ for some $\alpha$.
\end{itemize}
\end{lem}
\begin{proof} \emph{i)} Let $X_1$, $X_2$ be closed subgroups of $X$ with $X_2$ normal in $X_1$ and $X_1/X_2\cong\Sigma$. Consider the composite map
\[ X_1\cap Y \;\longmono\; X_1 \;\longepi\; X_1/X_2, \]
which has $X_2\cap Y$ as its kernel. Then it is easy to see that $(X_1\cap Y)/(X_2\cap Y)$ is a normal subgroup of $X_1/X_2$, which is simple by assumption. Hence there are two possibilities:
\begin{itemize}
\item[{\bf 1.}] $(X_1\cap Y)\big/(X_2\cap Y)\cong X_1/X_2\cong\Sigma$: in this case $\Sigma$ occurs in $Y$;
\item[{\bf 2.}] $(X_1\cap Y)\big/(X_2\cap Y)=0$, i.e. $X_1\cap Y=X_2\cap Y$: then $\Sigma$ occurs in $X/Y$ because
\[ \Sigma\cong X_1/X_2\cong(X_1/X_1\cap Y)\big/(X_2/X_2\cap Y)\cong(X_1Y/Y)\big/(X_2Y/Y). \]
\end{itemize}

\emph{ii)} The ``if'' part is easy. For the other implication, observe that the family $\{\Omega_\alpha\}_\alpha$ is a basis of neighbourhoods of the identity in $X$. Let $X_1,X_2\subset X$ be closed subgroups with $X_2$ normal in $X_1$ and $X_1/X_2\cong\Sigma$. Since $\Sigma$ is finite, $X_2$ is open in $X_1$, hence there exists an index $\alpha$ such that $X_1\cap\Omega_\alpha\subset X_2$ and, in particular, $X_1\cap\Omega_\alpha=X_2\cap\Omega_\alpha$. There are closed injections
\[ X_2/(X_2\cap\Omega_\alpha)\;\longmono\;X_1/(X_1\cap\Omega_\alpha)\;\longmono\;X/\Omega_\alpha \]
with $X_2/(X_2\cap\Omega_\alpha)$ normal in $X_1/(X_1\cap\Omega_\alpha)$, hence $\Sigma$ occurs in $X/\Omega_\alpha$. \end{proof}
As usual, set $PSL_2(\F_\ell):=SL_2(\F_\ell)/\{\pm1\}$ for all primes $\ell$. It is well known that these groups are simple for $\ell\geq5$ (see, e.g., \cite[Ch. XIII, Theorem 8.4]{la2}), and it is easy to see that they are pairwise non-isomorphic; a common theme of the next proposition and of the proof of Proposition \ref{sl2-pro} below will be the study of their occurrences in suitable profinite groups.
\begin{pro} \label{vert-control-pro}
The group $PSL_2(\F_\ell)$ occurs in $\gal{F(E[\ell^\infty])}{F}$ for almost all $\ell\not=p$.
\end{pro}
\begin{proof} We know that
\[ \gal{F(E[\ell^\infty])}{F}=\varprojlim_n\gal{F(E[\ell^n])}{F}. \]
Moreover, by Proposition \ref{hor-control-pro} the group $\gal{F(E[\ell])}{F}$ contains $SL_2(\F_\ell)$ for almost all primes $\ell\not=p$. Hence $PSL_2(\F_\ell)$ occurs in $\gal{F(E[\ell])}{F}$ for almost all $\ell\not=p$, and the claim follows from part \emph{ii)} of Lemma \ref{occurrence-lem}. \end{proof}

\subsection{Conclusion of the proof} \label{conclusion-subsec}

As remarked in the introduction, this is really an exercise in abstract group theory; in particular, elliptic curves will play no role. We follow \emph{mutatis mutandis} the exposition in \cite[Ch. 17, \S 5]{la} and \cite[Ch. IV, \S 3.4]{se2}. We begin by stating a very useful lemma, a proof of which is given in \cite[Ch. IV, \S 3.4, Lemmas 2 and 3]{se2}.
\begin{lem} \label{sl2-lem}
Let $\ell\geq5$ be a prime number and let $H$ be a closed subgroup of $SL_2(\Z_\ell)$ whose reduction $\bmod\;\ell$ surjects onto $PSL_2(\F_\ell)$. Then $H=SL_2(\Z_\ell)$.
\end{lem}
With this result at our disposal, we can prove
\begin{pro} \label{sl2-pro}
With self-explaining notation, the group $\rho_E(G_F)$ contains
\[ S_\ell:=\bigl(\dots,1,1,SL_2(\Z_\ell),1,1,\dots\bigr) \]
for almost all primes $\ell\not=p$.
\end{pro}
\begin{proof} By Proposition \ref{vert-control-pro}, we know that $PSL_2(\F_\ell)$ occurs in the component of $\rho_E(G_F)$ corresponding to $\ell$ for almost all primes $\ell\not=p$. To prove the proposition, we first show that $PSL_2(\F_\ell)$ occurs in $\rho_E(G_F)\cap S_\ell$ for almost all $\ell\not=p$. Let
\[ U_\ell:=\bigl(\dots,1,1,\hat\Gamma_\ell,1,1,\dots\bigr) \]
for all primes $\ell\not=p$. Clearly, for every prime $\ell\not=p$ there is an injection
\begin{equation} \label{rho-inclusion-eq}
\rho_E(G_F)\big/\bigl(\rho_E(G_F)\cap U_\ell\bigr)\;\longmono\;\hat\Gamma/U_\ell.
\end{equation}
By part \emph{ii)} of Lemma \ref{occurrence-lem}, $PSL_2(\F_\ell)$ does not occur in $\hat\Gamma_q$ for any prime $q\not=\ell$ if $\ell>5$. Hence, by \eqref{rho-inclusion-eq}, $PSL_2(\F_\ell)$ does not occur in the quotient $\rho_E(G_F)/(\rho_E(G_F)\cap U_\ell)$, so part \emph{i)} of Lemma \ref{occurrence-lem} ensures that it occurs in $\rho_E(G_F)\cap U_\ell$ for almost all $\ell\not=p$. It follows from part \emph{i)} of Lemma \ref{occurrence-lem} that for any such $\ell$ the group $PSL_2(\F_\ell)$ occurs in $\rho_E(G_F)\cap S_\ell$, which is closed in $S_\ell$ and maps to $PSL_2(\F_\ell)$ by reducing $\bmod\;\ell$ and projecting. Denote $M_\ell$ the image of $\rho_E(G_F)\cap S_\ell$ in $PSL_2(\F_\ell)$: we claim that $M_\ell=PSL_2(\F_\ell)$. If not, $M_\ell$ is a proper subgroup, so $PSL_2(\F_\ell)$ occurs in the kernel of this map, hence in
\begin{equation} \label{kernel-eq}
\bigl\{u\in SL_2(\Z_\ell)\mid u\equiv1\pmod{\ell}\bigr\}.
\end{equation}
But this is impossible if $\ell\geq5$ because the group in \eqref{kernel-eq} is prosolvable\footnote{This can be seen as follows. For all $n\geq1$ and $j\in\{1,\dots,n\}$ define the groups
\[ K^{(n)}_j:=\ker\Big(SL_2(\Z/\ell^n\Z)\;\longepi\;SL_2(\Z/\ell^j\Z)\Big). \]
Then for all $n$ there is a chain
\[ K^{(n)}_n=\{1\}\subset K^{(n)}_{n-1}\subset\dots\subset K^{(n)}_1 \]
with $K^{(n)}_{j+1}$ normal in $K^{(n)}_j$ and $K^{(n)}_j/K^{(n)}_{2j}$ abelian (for example, it injects in the additive group $M_2(\Z/\ell^n\Z)$). This shows that $K^{(n)}_1$ is solvable for all $n\geq1$, and since
\[ \bigl\{u\in SL_2(\Z_\ell)\mid u\equiv1\pmod{\ell}\bigr\}=\varprojlim_n K^{(n)}_1 \]
the claim follows.} while $PSL_2(\F_\ell)$ is nonabelian and simple, hence nonsolvable.

Therefore $\rho_E(G_F)\cap S_\ell$ maps onto $PSL_2(\F_\ell)$, hence $\rho_E(G_F)\cap S_\ell\cong SL_2(\Z_\ell)$ by Lemma \ref{sl2-lem}. \end{proof}
\begin{cor} \label{sl2-cor}
There exists a finite set $S$ of prime numbers such that $p\in S$ and $\rho_E(G_F)$ contains $\prod_{\ell\notin S}SL_2(\Z_\ell)$.
\end{cor}
In the statement of this corollary, the partial product is understood as a subgroup of $\hat\Gamma$ in the natural way.
\begin{proof} With identifications as before, by Proposition \ref{sl2-pro} there exists a finite set $S$ of primes such that $p\in S$ and $\rho_E(G_F)$ contains $SL_2(\Z_\ell)$ for all $\ell\notin S$. It follows that $\rho_E(G_F)$ contains
\[ \mathscr S:=\bigcup_{\substack{|T|<\infty\\T\,\cap\,S=\emptyset}}\prod_{\ell\in T}SL_2(\Z_\ell), \]
where $T$ runs through the finite sets of primes that are disjoint from $S$. But $\rho_E(G_F)$ is closed in $\hat\Gamma$, hence it contains the closure of $\mathscr S$, which is the product appearing in the statement of the corollary. \end{proof}
Now we are in a position to prove Igusa's theorem. For the reader's convenience we restate Theorem \ref{igusa-thm}, and then proceed to its proof.
\begin{thm}[Igusa] \label{main-thm}
The group $\rho_E(G_F)$ is open in $\hat\Gamma$.
\end{thm}
\begin{proof} Let $S$ be as in Corollary \ref{sl2-cor}, let $S':=S-\{p\}$ and let $S''$ be the (infinite) set of primes not in $S$. As before, write
\[ \hat\Gamma_{S'}:=\prod_{\ell\in S'}\hat\Gamma_\ell, \qquad \hat\Gamma_{S''}:=\prod_{\ell\in S''}\hat\Gamma_\ell. \]
Denote $\rho_{E,S'}(G_F)$ and $\rho_{E,S''}(G_F)$ the projections of $\rho_E(G_F)$ to $\hat\Gamma_{S'}$ and $\hat\Gamma_{S''}$, respectively. A combination of Corollary \ref{sl2-cor} and part \emph{i)} of Lemma \ref{surj-det-lem} shows that $\rho_{E,S''}(G_F)=\hat\Gamma_{S''}$. On the other hand, $\rho_{E,S'}(G_F)$ is open in $\hat\Gamma_{S'}$ by Proposition \ref{open-pro}. It follows that
\[ \rho_E(G_F)\supset\rho_{E,S'}(G_F)\times\hat\Gamma_{S''}, \]
which is an open subgroup of $\hat\Gamma$. The theorem is proved. \end{proof}
\begin{rem}
To prove Igusa's theorem one could also proceed as follows. By an argument with Lie algebras exactly as in \cite[Ch. IV, \S 3.4, Lemma 6]{se2}, it can be deduced from Proposition \ref{open-pro} and Corollary \ref{sl2-cor} that $\rho_E(G_F)$ contains an open subgroup of $\prod_{\ell\not=p}SL_2(\Z_\ell)$. But then part \emph{ii)} of Lemma \ref{surj-det-lem} allows one to conclude the proof as in Proposition \ref{hor-control-pro2}.
\end{rem}

\section{An arithmetic application} \label{applications-sec}

In this final section we collect an arithmetic consequence of Theorem \ref{igusa-thm}. We retain throughout our previous notation; in particular, $F=\F_r(\cC)$ is a function field of characteristic $p>0$ and $F^s$ is the separable closure of $F$ contained in an algebraic closure $\bar{F}$. We remark that Theorem \ref{igusa-thm} is valid in any positive characteristic, though in the present paper we have proved it only for $p>3$.

\subsection{Main application} \label{main-application-subsec}

The result we want to prove in this $\S$ says that a non-isotrivial elliptic curve $E_{/F}$ has only finitely many torsion points rational over abelian extensions of $F$. Although properties in the same spirit have been exploited, at least implicitly, in various recent works on the arithmetic of elliptic curves in positive characteristic (cf., e.g., \cite{bl}, \cite{bre}, \cite{bro}, \cite{vi}), it seems that (quite surprisingly) the result below has never been written down in detail.

We begin with a lemma in linear algebra.
\begin{lem} \label{commutator-lem}
Let $\ell$ be a prime and let $S_n$ be the kernel of the reduction-modulo-$\ell^n$ map
\[ SL_2(\Z_\ell)\longrightarrow SL_2(\Z/\ell^n\Z). \]
The commutator subgroup $[S_n,S_n]$ contains $S_{2n+2}$ for all $n\geq2$.
\end{lem}
\begin{proof}[Sketch of proof.] One just adapts the arguments of \cite[Ch. XIII, Lemma 8.1]{la2} (as in Lemma \ref{H-open-lem}). \end{proof}
Now we can prove the result we alluded to before.
\begin{thm} \label{main-application-thm}
Let $E_{/F}$ be a non-isotrivial elliptic curve and let $H:=F^{ab}$ be the maximal abelian extension of $F$. The group $E_\mathrm{tors}(H)$ is finite.
\end{thm}
\begin{proof} Define an \emph{Igusa prime} to be a prime number satisfying the second part of Theorem \ref{igusa2-thm}. To prove our result we show that
\begin{itemize}
\item[\emph{i)}] $E[\ell^\infty](H)$ is finite for all primes $\ell$;
\item[\emph{ii)}] $E[\ell](H)=\{0\}$ if $\ell$ is an Igusa prime.
\end{itemize}
Since by Theorem \ref{igusa2-thm} almost all primes are Igusa, the theorem will follow.

\emph{i)} We need to distinguish between two cases according as whether $\ell$ is equal to the characteristic of $F$ or not.

If $\ell=p$ the claim is immediate from Proposition \ref{sep-torsion-pro} (see also \cite[Lemma 2.2]{bre} for a proof using a different argument).

If $\ell\not=p$ define the groups $S_n$ as in Lemma \ref{commutator-lem}. By Theorem \ref{igusa2-thm}, $S_n\subset\rho_{E,\ell}(G_F)$ for some $n\geq2$. Since $\gal{F^s}{H}$ is the topological closure of the commutator subgroup $[G_F,G_F]$, its image under $\rho_{E,\ell}$ contains the commutator subgroup $[S_n,S_n]$ of $GL_2(\Z_\ell)$, and hence, by Lemma \ref{commutator-lem}, $S_{2n+2}$. Therefore $E[\ell^\infty](H)$ is contained in the fixed subgroup of $E[\ell^\infty]$ under the action of $S_{2n+2}$, which is the finite group $E[\ell^{2n+2}]$.

\emph{ii)} Let $\ell$ be an Igusa prime and set $F_\ell:=F(E[\ell])$; then $\gal{F_\ell}{F}$ contains a subgroup isomorphic to $SL_2(\F_\ell)$, hence the Galois orbit of a nonzero point $P\in E[\ell]$ is the whole $E[\ell]-\{0\}$. In particular, since the extension $H/F$ is normal, if $P\in E[\ell](H)$ and $P\not=0$ then $F_\ell\subset H$, which is impossible because $H/F$ is abelian. Thus $E[\ell](H)=\{0\}$, and the theorem is completely proved. \end{proof}
\begin{rem}
Replacing Igusa's theorem with Serre's theorem (Theorem \ref{serre-thm}) and disregarding, of course, the ``$\ell=p$'' part, the proof of Theorem \ref{main-application-thm} carries over \emph{verbatim} to the case of an elliptic curve without complex multiplication defined over a number field. More precisely, one shows that if $K$ is a number field and $E_{/K}$ is a non-CM elliptic curve then there are only finitely many torsion points on $E$ that are rational over abelian extensions of $K$. Note that this is the one-dimensional case of a theorem of Zarhin (\cite[Theorem 1]{z}) for non-CM abelian varieties (see also \cite[Theorem 1]{ri} for a weaker result which is valid for all abelian varieties).
\end{rem}
\begin{rem}
If $\Z\subsetneq\text{End}(E)$, i.e., if $E_{/F}$ is isotrivial (resp., has complex multiplication) in the function field (resp., in the number field) case, then Theorem \ref{main-application-thm} is false. Indeed, with notation as in the introduction, it can be shown that there exists a finite extension $K/F$ such that
\[ E_\text{$(p)$-tors}\subset E(K^{ab}). \]
This fact is a consequence of the results described in Appendix \ref{isotrivial-sec} in the function field case and of the theory of complex multiplication in the number field case.
\end{rem}

\appendix

\section{The isotrivial case} \label{isotrivial-sec}

For the sake of completeness, in this appendix we treat the case of isotrivial elliptic curves. We remark that we only give a ``qualitative'' description of the image of Galois; actually, something more precise can presumably be proved, but we shall not pursue this issue here.

So let $E_{/F}$ be our elliptic curve over the function field $F=\F_r(\cC)$ and suppose that $E$ is isotrivial. Recall that this means that after a finite extension $L/F$ the curve $E$ becomes isomorphic to an elliptic curve $E'$ defined over $\F_r$; equivalently, $j(E)\in\F_r$. By Lemma \ref{end-lem}, we can also equivalently define the elliptic curve $E_{/F}$ to be isotrivial if its ring of endomorphisms is larger than $\Z$.

First of all, note that there is an inclusion $\rho_E(G_F)\subset\hat\Gamma$ (in fact, the arguments in \S \ref{weil-subsec} do not rely on $E$ being non-isotrivial, but just on general properties of the Weil pairing). Our present goal is to show that $\rho_E(G_F)$ is not open in $\hat\Gamma$ (in particular, the above inclusion is proper), so that Theorem \ref{igusa2-thm} is always false in the isotrivial case.

Let $L/F$ be an extension as above, so that the base-changed curve $E_{/L}$ is isomorphic to an elliptic curve $E'$ defined over $\F_r$. Since we are assuming that $p>3$, the extension $L/F$ may be taken to be separable, and we denote $G_L\subset G_F$ the absolute Galois group of $L$. Observe that there are isomorphisms
\[ E_\text{$(p)$-tors}(F^s)=E_\text{$(p)$-tors}(L^s)\cong E'_\text{$(p)$-tors}(\bar\F_r) \]
of $G_L$-modules, and if we set $\F_s:=L\cap\bar\F_r$ (that is, $\F_s$ is the field of constants of $L$) then the action of $G_L$ on $E'_\text{$(p)$-tors}$ factors through the absolute Galois group $G_{\mathbb F_s}:=\gal{\bar\F_r}{\F_s}$. This last group is procyclic, isomorphic to the profinite completion $\hat\Z$ of $\Z$. It follows that $\rho_E(G_L)$ is a procyclic (hence abelian) group, so it cannot be open in $\hat\Gamma$. Since $\rho_E(G_L)$ has finite index in $\rho_E(G_F)$, we can state the following
\begin{pro} \label{isotrivial-pro}
If $E_{/F}$ is isotrivial then $\rho_E(G_F)$ is not open in $\hat\Gamma$.
\end{pro}
In fact, we can say something more. To this end, define the profinite groups $\hat{H}'$ and $\hat\Gamma'$ as in \eqref{gamma-hat-eq} by replacing $r$ with $s$. Clearly, $\hat{H}'\subset\hat{H}$ and $\hat\Gamma'\subset\hat\Gamma$. Now recall that $\hat\Z_{(p)}$ is a shorthand for $\prod_{\ell\not=p}\Z_\ell$ and let
\[ \rho_{E'}:G_{\mathbb F_r}\longrightarrow GL_2(\hat\Z_{(p)}) \]
be the Galois representation attached to $E'$. The cyclotomic character $\chi$ induces an isomorphism
\[ \chi: G_{\mathbb F_s}\overset{\cong}{\longrightarrow}\hat{H}' \]
with the property that $\det\circ\rho_{E'}=\chi$, thus the diagram
\[ \xymatrix@C=35pt{G_L\ar@{>>}[r]^-{\rho_E}\ar@{>>}[d] & \rho_E(G_L)\ar[d]^-{\det}\\
             G_{\mathbb F_s}\ar[ur]^-{\rho_{E'}}\ar[r]^\chi & \hat{H}'} \]
is commutative. It follows that the determinant gives an isomorphism
\[ \det:\rho_E(G_L)\overset{\cong}{\longrightarrow}\hat{H}', \]
hence the short exact sequence defining $\hat\Gamma'$ admits a (topological) splitting as follows:
\[ \xymatrix@R=15pt{& & \rho_E(G_L)\ar@{^{(}->}[d]\ar[dr]^-{\det} & &\\
          0\ar[r] & SL_2(\hat\Z_{(p)})\ar[r] & \hat\Gamma'\ar[r] & \hat{H}'\ar[r]\ar@/^1pc/@{-->}[l] & 0.} \]
In other words, we have proved the following
\begin{thm} \label{isotrivial-thm}
With notation as above, if $E_{/F}$ is isotrivial then there are isomorphisms
\[ \hat\Gamma'\cong SL_2(\hat\Z_{(p)})\rtimes\hat{H}'\cong SL_2(\hat\Z_{(p)})\rtimes\rho_E(G_L) \]
of topological groups. In particular, $\rho_E(G_L)$ is not open in $\hat\Gamma'$.
\end{thm}
Actually, by working componentwise it can be shown (as above) that $\rho_{E,\ell}(G_F)\subsetneq\hat\Gamma_\ell$ for all primes $\ell\not=p$.
\begin{rem}
The goal of this appendix was to highlight the following ``principle'': as long as one is interested in the ``asymptotic size'' of the images of the Galois representations on Tate modules of elliptic curves, \emph{isotriviality} is the counterpart in characteristic $p$ of \emph{complex multiplication} over number fields. In fact, when the ring of endomorphisms is ``as small as possible'' (i.e., equal to $\Z$) the image of Galois is definitively ``as large as possible'' (i.e., equal to $GL_2(\Z_\ell)$ in characteristic zero and to $\hat\Gamma_\ell$ in positive characteristic), while this never happens (both in positive characteristic and in characteristic zero, cf. Remark \ref{serre-rem}) when the elliptic curve has an endomorphism ring of rank greater than one.
\end{rem}


\begin{thebibliography}{11}

\bibitem{bl} \textsc{A. Bandini, I. Longhi}, Control theorems for elliptic curves over function fields, {\it Int. J. Number Theory}, to appear.
\bibitem{bo} \textsc{G. B\"ockle}, Arithmetic over function fields: a cohomological approach. In {\it Number fields and function fields -- two parallel worlds}, G. van der Geer, B. Moonen and R. Schoof (eds.), Progress in Mathematics {\bf 239},  Birkh\"auser, Boston, 2005, 1-38.
\bibitem{bre} \textsc{F. Breuer}, Higher Heegner points on elliptic curves over function fields, {\it J. Number Theory} {\bf  104} (2004), 315-326.
\bibitem{bro} \textsc{M. L. Brown}, {\it Heegner modules and elliptic curves}, Lecture Notes in Mathematics {\bf 1849}, Springer, Berlin, 2004.
\bibitem{ch} \textsc{A. C. Cojocaru, C. Hall}, Uniform results for Serre's theorem for elliptic curves, {\it Int. Math. Res. Not.} {\bf 50} (2005), 3065-3080.
\bibitem{d} \textsc{P. Deligne}, Preuve des conjectures de Tate et de Shafarevitch (d'apr\`es G. Faltings). In ``S\'eminaire Bourbaki'' 36e ann\'ee, 1983/84, no. 619, {\it Ast\'erisque} {\bf 121-122} (1985), 25-41.
\bibitem{dr} \textsc{P. Deligne, M. Rapoport}, Les sch\'emas de modules de courbes elliptiques. In {\it Modular functions II}, P. Deligne and W. Kuyk (eds.), Lecture Notes in Mathematics {\bf 349}, Springer-Verlag, Berlin, 1973, 143-316.
\bibitem{fa} \textsc{G. Faltings}, Endlichkeitss\"atze f\"ur abelsche Variet\"aten \"uber Zahlk\"orpern, \emph{Invent. Math.} {\bf 73} (1983), 349-366.
\bibitem{f1} \textsc{G. Frey}, Links between solutions of $A-B=C$ and elliptic curves. In {\it Number theory}, H. P. Schlickewei and E. Wirsing (eds.), Lecture Notes in Mathematics {\bf 1380}, Springer-Verlag, New York, 1989, 31-62.
\bibitem{f2} \bysame, On ternary equations of Fermat type and relations with elliptic curves. In {\it Modular forms and Fermat's last theorem}, G. Cornell, J. H. Silverman and G. Stevens (eds.), Springer-Verlag, New York, 1997, 527-548.
\bibitem{f3} \bysame, Galois representations attached to elliptic curves and Diophantine problems. In {\it Number theory}, M. Jutila and T. Mets\"ankyl\"a (eds.), Walter de Gruyter \& Co., Berlin, 2001, 71-104.
\bibitem{hs} \textsc{M. Hindry, J. H. Silverman}, The canonical height and integral points on elliptic curves, {\it Invent. Math.} {\bf 93} (1988), 419-450.
\bibitem{i} \textsc{J.-I. Igusa}, Fibre systems of jacobian varieties (III. Fibre systems of elliptic curves), {\it Amer. J. Math.} {\bf 81} (1959), 453-476.
\bibitem{la} \textsc{S. Lang}, {\it Elliptic functions}, second edition, Graduate Texts in Mathematics {\bf 112}, Springer-Verlag, New York, 1987.
\bibitem{la2} \bysame, {\it Algebra}, revised third edition, Graduate Texts in Mathematics {\bf 211}, Springer-Verlag, New York, 2002.
\bibitem{li} \textsc{Q. Liu}, {\it Algebraic geometry and arithmetic curves}, Oxford Graduate Texts in Mathematics {\bf 6}, Oxford University Press, Oxford, 2002.
\bibitem{mu} \textsc{D. Mumford}, {\it Abelian varieties}, Tata Institute of Fundamental Research studies in mathematics {\bf 5}, Oxford University Press, Oxford, 1970.
\bibitem{ri} \textsc{K. A. Ribet}, Torsion points on abelian varieties in cyclotomic extensions (appendix to an article by N. Katz and S. Lang), \emph{Enseign. Math. (2)} {\bf 27} (1981), 315-319.
\bibitem{r} \textsc{M. Rosen}, {\it Number theory in function fields}, Graduate Texts in Mathematics {\bf 210}, Springer-Verlag, New York, 2002.
\bibitem{se3} \textsc{J.-P. Serre}, Facteurs locaux des fonctions z\^eta des vari\'et\'es alg\'ebriques (d\'efinitions et conjectures). In {\it S\'eminaire Delange-Pisot-Poitou. Th\'eorie des nombres}, tome {\bf 11}, n. 2 (1969-1970), exp. n. 19, 1-15 (\OE uvres n. {\bf 87}). Available at \texttt{http://www.numdam.org}.
\bibitem{se1} \bysame, Propri\'et\'es galoisiennes des points d'ordre fini des courbes elliptiques, {\it Invent. Math.} {\bf 15} (1972), 259-331.
\bibitem{se2} \bysame, {\it Abelian $l$-adic representations and elliptic curves}, revised second edition, Research Notes in Mathematics {\bf 7}, A K Peters, Wellesley, MA, 1998.
\bibitem{si1} \textsc{J. H. Silverman}, {\it The arithmetic of elliptic curves}, Graduate Texts in Mathematics {\bf 106}, Springer-Verlag, New York, 1986.
\bibitem{si2} \bysame, Heights and elliptic curves. In {\it Arithmetic Geometry}, G. Cornell and J. H. Silverman (eds.), revised second printing, Springer-Verlag, New York, 1998, 253-265.
\bibitem{si3} \bysame, {\it Advanced topics in the arithmetic of elliptic curves}, corrected second printing, Graduate Texts in Mathematics {\bf 151}, Springer-Verlag, New York, 1999.
\bibitem{sz} \textsc{L. Szpiro}, La conjecture de Mordell (d'apr\`es G. Faltings). In ``S\'eminaire Bourbaki'' 36e ann\'ee, 1983/84, no. 619, {\it Ast\'erisque} {\bf 121-122} (1985), 83-103.
\bibitem{u} \textsc{D. Ulmer}, Elliptic curves and analogies between number fields and function fields. In {\it Heegner points and Rankin $L$-series}, H. Darmon and S.-W. Zhang (eds.), MSRI Publications {\bf 49}, Cambridge University Press, Cambridge, 2004, 285-315.
\bibitem{vi} \textsc{S. Vigni}, On ring class eigenspaces of Mordell-Weil groups of elliptic curves over global function fields, {\it J. Number Theory}, to appear.
\bibitem{vo} \textsc{F. Voloch}, Explicit $p$-descent for elliptic curves in characteristic $p$, {\it Compos. Math.} {\bf 74} (1990), 247-258.
\bibitem{z} \textsc{Yu. G. Zarhin}, Endomorphisms and torsion of abelian varieties, {\it Duke Math. J.} {\bf 54} (1987), 131-145.

\end{thebibliography}
\end{document}